\documentclass[12pt,a4paper]{article}
 
\usepackage[latin1]{inputenc}
\usepackage[english]{babel}
\usepackage[T1]{fontenc}
\usepackage{amsmath} 
\usepackage{amsfonts}
\usepackage{amssymb}
\usepackage{amsthm}
\usepackage{dsfont}
\usepackage{stmaryrd}
\usepackage{tikz}
\usepackage{todonotes}
\usetikzlibrary{matrix,arrows.meta}
\usepackage[title,titletoc,toc]{appendix}

\allowdisplaybreaks

\usepackage{hyperref}
\usepackage{xcolor}
\usepackage[a4paper,top=2.0cm,bottom=2.50cm,left=2.0cm,right=2.0cm]{geometry} 

\theoremstyle{plain}
\newtheorem{theorem}{Theorem}[section]
\newtheorem{lemma}[theorem]{Lemma}
\newtheorem{proposition}[theorem]{Proposition} 
\newtheorem{corollary}[theorem]{Corollary}

\theoremstyle{definition} 
\newtheorem{remark}[theorem]{Remark}

\begin{document}
%
%
%
%
%
%
%

\title{{Shape differentials and sensitivity of the heat equation}}


\author{Matteo Dalla Riva\thanks{{Dipartimento di Tecnica e Gestione dei Sistemi Industriali, Universit\`a degli Studi di Padova, Stradella S. Nicola 3, 36100 Vicenza,  Italy. E-mail: {matteo.dallariva@unipd.it}}},
        Paolo Luzzini\thanks{Dipartimento di Scienze e Innovazione Tecnologica, Universit\`a degli Studi del Piemonte Orientale ``Amedeo Avogadro'', Viale Teresa Michel 11, 15121 Alessandria, Italy. E-mail: {paolo.luzzini@uniupo.it}}, 
        Paolo Musolino\thanks{Dipartimento di Matematica ``Tullio Levi-Civita'', Universit\`a degli Studi di Padova, Via Trieste 63, 35121 Padova, Italy. E-mail: {paolo.musolino@unipd.it}}}


\date{September 21, 2026}

\maketitle

\noindent
{\bf Abstract:} 
We  consider Dirichlet and Neumann problems for the heat equation in a domain containing an interior hole. The domain has a fixed outer boundary and a variable inner boundary determined by a diffeomorphism $\phi$. We analyze the maps that assign to the infinite-dimensional shape parameter $\phi$ the corresponding solutions, and we prove that they are smooth. Motivated by  applications to  inverse problems, we then compute the differentials with respect to $\phi$ of  the Dirichlet-to-Neumann and Neumann-to-Dirichlet maps.



\section{Introduction}

In this paper, we analyze the sensitivity of the solutions to the {Dirichlet and Neumann problems} for the heat equation with respect to perturbations of the spatial domain.

Specifically, {we first} consider a Dirichlet problem for the heat equation in a perturbed annular domain in {$\mathbb{R}^n$}, with a fixed exterior boundary $\partial \Omega^o$ and a moving interior boundary $\phi(\partial \Omega^i)$, where $\phi$ is a suitable diffeomorphism serving as a parameter of the problem.

{We} prove that the solution depends smoothly on the parameter $\phi$ and on the boundary data. We then compute the shape differential (i.e., from now on, the differential with respect to $\phi$) of the normal derivative of the solution {on the outer boundary}.

{We next consider} a Neumann problem for the heat equation in a perturbed annular domain of the same type. {In this case as well,} we prove that the solution depends smoothly on the parameter and on the boundary data, and we compute the corresponding shape differential.

This topic falls within the mathematical field known as {\em shape optimization}. For an in-depth treatment, we refer the reader to the monographs by Soko\l owski and Zol\'esio \cite{SoZo92}, Novotny and Soko\l owski \cite{NoSo13}, and Henrot and Pierre \cite{HePi18}, among others. The subject is motivated by practical applications, particularly in the design of objects optimized for specific functional purposes. Further applications arise in inverse problems, where the goal is to infer properties of an object---such as its shape---from indirect measurements of its behavior, for example, through wave scattering.

In both shape optimization and inverse problems, it is crucial to understand the regularity of the map that associates the shape of the domain with the corresponding solution. Different levels of regularity have different implications. For example, differentiability enables the use of differential calculus to characterize optimal configurations, while smoothness allows the solution to be approximated by Taylor polynomials to any desired degree of accuracy.

{The explicit formula that we derive in Section~\ref{s:inv} for the shape differential of the Neumann trace of the solution to the Dirichlet problem can be compared with a result of Chapko, Kress, and Yoon \cite{ChKrYo98}, where a similar expression was used to address the inverse problem of reconstructing the interior boundary of an annular domain of arbitrary shape from overdetermined Cauchy data prescribed on the exterior boundary. Our work extends that of Chapko, Kress, and Yoon \cite{ChKrYo98} in several respects. First, we establish $C^\infty$-smoothness rather than mere differentiability. Second, our approach applies in arbitrary spatial dimensions, rather than being restricted to the planar case. Third, we consider domains of class $C^{1,\alpha}$ instead of requiring $C^2$ regularity. {Moreover,} we are able to handle non-homogeneous boundary data on the boundary of the hole. Similarly, the results of Section~\ref{s:neum} concerning the dependence of the solution to the Neumann problem likewise generalize the results of Chapko, Kress, and Yoon \cite{ChKrYo99} in the same spirit.}

The analysis in the present paper is based on the Functional Analytic Approach introduced by Lanza de Cristoforis (see, e.g., \cite{DaLaMu21} and the references therein) and, as such, relies heavily on potential-theoretic methods. In particular, as in \cite{DaLuMoMu25}, we exploit the results of \cite{DaLu23} concerning the smooth dependence of the heat single-layer potential on the shape of its integration support. In \cite{DaLuMoMu25}, the goal was to prove the smooth dependence of the solution to a nonlinear mixed problem for the heat equation under domain perturbations. Here, instead, we consider the {Dirichlet and Neumann problems} for the heat equation and analyze the maps that associate with the infinite-dimensional shape parameter the corresponding solution and its normal derivative. We also compute the shape differential of the normal derivative with respect to domain perturbations. {Our main goal is to derive explicit formulas for these shape differentials and to extend in several directions the results of Chapko, Kress, and Yoon \cite{ChKrYo98,ChKrYo99}. We emphasize that the method developed here can be adapted to the analysis of a broad class of boundary value problems{, including problems of a non-variational nature}.}

{Our analysis builds upon previous results concerning the smooth dependence of layer potentials on domain perturbations. In particular, we use the results of \cite{DaLu23}, which establish the regularity of heat layer potentials under perturbations of the spatial support in the optimal H\"older setting. We note that such results are relatively scarce for parabolic operators, whereas the corresponding elliptic theory is much more developed. In this context, we mention, for example, the work of Potthast \cite{Po94,Po96a}, where it is proved that the layer potentials for the Helmholtz equation are Fr\'echet differentiable with respect to the shape of the integration support. Similar results have been obtained for a variety of equations, including the Stokes system of fluid dynamics and the Lam\'e equations of elasticity. We refer, for instance, to the works of Charalambopoulos \cite{Ch95}, Costabel and Le Lou\"er \cite{CoLe12a}, Haddar and Kress \cite{HaKr04}, Hettlich \cite{He95}, and Kirsch \cite{Ki93}.}

Few results establish regularity properties beyond differentiability. {Examples include} the works of Lanza de Cristoforis and his collaborators: analyticity results for harmonic layer potentials can be found in \cite{LaRo04}, with an extension to more general elliptic operators in \cite{DaLa10}; smoothness results for heat layer potentials are established in \cite{DaLu23} and, in the periodic setting, in \cite{DaLuMoMu24}.

Another notable line of research is represented by {works} on ``shape holomorphy.'' As examples, we mention the work of Henr\'iquez and Schwab \cite{HeSc21} on the Calder\'on projector for the Laplacian in $\mathbb{R}^2$, as well as the papers by Pinto, Henr\'iquez, and Jerez-Hanckes \cite{PiHeJe24} and by D\"olz and Henr\'iquez \cite{DoHe24} on boundary integral operators.

In contrast with the vast literature on elliptic problems, significantly fewer results are available on shape sensitivity for parabolic equations. Besides the previously mentioned works \cite{DaLu23,DaLuMoMu24,DaLuMoMu25}, the contributions of Chapko, Kress, and Yoon \cite{ChKrYo98,ChKrYo99}, as well as Hettlich and {Rundell} \cite{HeRu01} and {Harbrecht and Tausch \cite{HaTa11,HaTa13}}, establish Fr\'echet differentiability of the solution with respect to domain perturbations and explore applications to inverse { and shape optimization} problems in heat conduction.

\subsection*{Main results} We now introduce the boundary value {problems} considered in this paper. To this end, we fix $\alpha \in \mathopen]0,1[$ and a natural number
\[
n \in \mathbb{N} \setminus \{0,1\}\, .
\]
We then choose two sets $\Omega^{o}$ and $\Omega^{i}$ satisfying the following conditions:  
{\begin{equation}\label{introsetconditions}
	\begin{split} 
		&\mbox{$\Omega^{o}$ and $\Omega^{i}$ are bounded, open, connected subsets of $\mathbb{R}^n$ of class $C^{1,\alpha}$,}
		\\
		&\mbox{with connected exteriors $(\Omega^{o})^{-} := \mathbb{R}^n\setminus \overline{\Omega^{o}}$ and $(\Omega^{i})^- := \mathbb{R}^n\setminus \overline{\Omega^{i}}$.}
	\end{split}
\end{equation}}
Here, the symbol $\overline{\cdot}$ denotes the closure of a set. For the definition of Schauder spaces and domains of class $C^{1,\alpha}$, we refer to Gilbarg and Trudinger \cite[pp.~52, 95]{GiTr83}. Moreover, as in \eqref{introsetconditions}, for an open set $\tilde{\Omega}\subseteq\mathbb{R}^n$, we denote its exterior $\mathbb{R}^n \setminus \overline{\tilde{\Omega}}$ by $\tilde{\Omega}^-$.

We consider our boundary value problems on the product of a bounded time interval $[0,T]$ with a perforated domain obtained by removing from $\Omega^{o}$ a perturbed copy of $\Omega^{i}$. More precisely, we perturb $\Omega^{i}$ by means of a diffeomorphism $\phi$ from $\partial\Omega^{i}$ onto its image belonging to the class
\[
\mathcal{A}_{\partial\Omega^{i}} := \left\{ \phi \in  C^{1,\alpha}(\partial\Omega^{i}, \mathbb{R}^n): \, \phi \text{ injective}, \, d\phi(y) \text{ injective for all } y \in \partial\Omega^{i} \right\}\,.
\]
The Jordan--Leray Separation Theorem (cf.~Deimling \cite[Thm.~5.2, p.~26]{De85}) ensures that $\phi(\partial\Omega^{i})$ divides $\mathbb{R}^n$ into exactly two open connected components, one bounded and {the other} unbounded. We denote by
\[
\Omega^{i}[\phi]
\]
the bounded open connected component of $\mathbb{R}^n \setminus \phi(\partial\Omega^{i})$. Clearly,
\[
\partial \Omega^{i}[\phi]=\phi(\partial\Omega^{i})\,.
\]
Since we require {$\overline{\Omega^{i}[\phi]}$ to be} contained in $\Omega^{o}$ in order for it to represent a hole, we introduce the set
\[
\mathcal{A}^{\Omega^{o}}_{\partial\Omega^{i}} := \left\{ \phi \in \mathcal{A}_{\partial\Omega^{i}} : \overline{\Omega^{i}[\phi]} \subseteq \Omega^{o}  \right\}\,.
\]
The sets $\mathcal{A}_{\partial\Omega^{i}}$ and $\mathcal{A}^{\Omega^{o}}_{\partial\Omega^{i}}$ are well-known to be open  in $C^{1,\alpha}(\partial\Omega^i,\mathbb{R}^n)$ (see, e.g., Lanza de Cristoforis and Rossi  \cite[Lem. 2.5, p. 143]{LaRo04}). If $\phi\in \mathcal{A}^{\Omega^{o}}_{\partial\Omega^{i}}$, then we can consider the perturbed perforated domain $\Omega^{o}\setminus\overline{\Omega^{i}[\phi]}$ (see Figure~\ref{fig:1}).

\begin{figure}[!htb]
\centering
\includegraphics[width=4.2in]{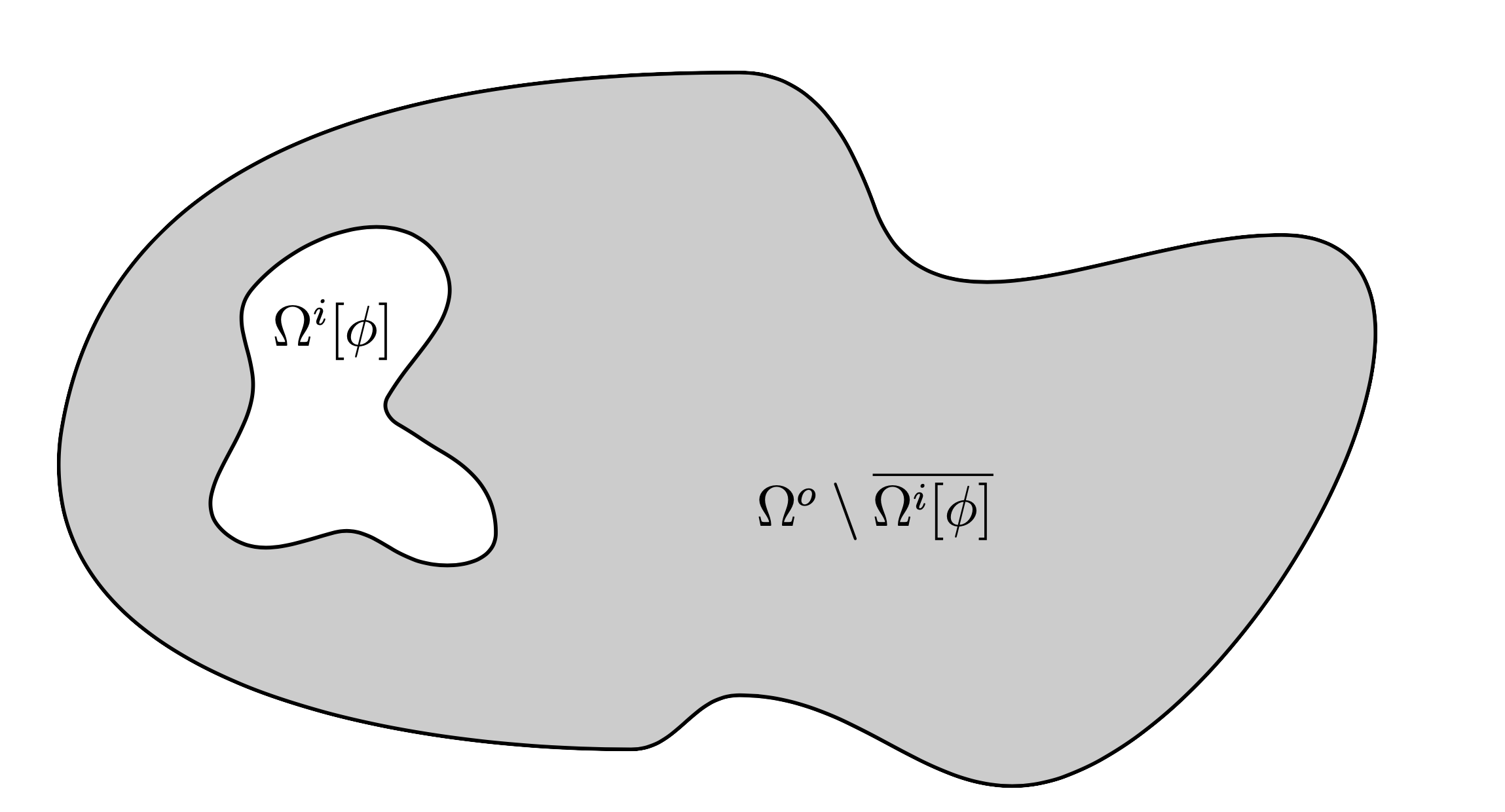}
\caption{\textit{The sets $\Omega^{o}\setminus\overline{\Omega^{i}[\phi]}$ and $\Omega^{i}[\phi]$.}}
\label{fig:1}
\end{figure}

We observe that the boundary $\partial(\Omega^{o}\setminus \overline{\Omega^{i}[\phi]})$ has two connected components: the fixed boundary $\partial\Omega^{o}$ and the moving boundary $\partial \Omega^{i}[\phi]=\phi(\partial \Omega^{i})$. Accordingly, to prescribe Dirichlet boundary conditions on $[0,T]\times\partial(\Omega^{o}\setminus \overline{\Omega^{i}[\phi]})$, we consider a pair of functions
\begin{equation}\label{introfunconditions}
(g^o,g^i) \in C_{0}^{\frac{1+\alpha}{2}; 1+\alpha}([0,T] \times \partial \Omega^{o}) \times C_{0}^{\frac{1+\alpha}{2}; 1+\alpha}([0,T] \times \partial \Omega^{i})\, .
\end{equation}
We refer to Section~\ref{s:prel} for the definition of the parabolic Schauder spaces.

We then consider the following Dirichlet boundary value problem for a function
\[
u \in C_{0}^{\frac{1+\alpha}{2}; 1+\alpha}([0,T] \times (\overline{\Omega^{o}} \setminus \Omega^{i}[\phi])):
\]

\begin{equation}\label{princeqpertu}
\begin{cases}
    \partial_t u - \Delta u = 0 & \quad\text{in } ]0,T] \times (\Omega^{o} \setminus \overline{\Omega^{i}[\phi]}),
    \\
    u(t,x) = g^o(t,x) & \quad \forall (t,x)\in [0,T] \times \partial \Omega^{o},
    \\
    u(t,x) = g^i(t,\phi^{-1}(x)) & \quad \forall (t,x)\in [0,T] \times \partial \Omega^{i}[\phi],
    \\
    u(0,\cdot)=0 & \quad \text{in } \overline{\Omega^{o}} \setminus \Omega^{i}[\phi]\, .
\end{cases}
\end{equation}

As is well known, for every triple
\[
(\phi,g^o, g^i ) \in \mathcal{A}^{\Omega^{o}}_{\partial\Omega^{i}}\times   C_0^{\frac{1+\alpha}{2};  1+\alpha}([0,T] \times \partial\Omega^{o}) \times C_0^{\frac{1+\alpha}{2};1+\alpha}([0,T] \times \partial\Omega^{i}),
\]
problem \eqref{princeqpertu} admits a unique solution in
\[
C_{0}^{\frac{1+\alpha}{2}; 1+\alpha}([0,T] \times (\overline{\Omega^{o}} \setminus \Omega^{i}[\phi])),
\]
which we denote by $u^D_{\phi,g^o,g^i}$. Our aim is to study the dependence of the solution $u^D_{\phi,g^o,g^i}$, and of associated quantities, on the triple $(\phi,g^o,g^i)$.

We note, however, that for fixed $\phi$, the map
\[
(g^o, g^i) \longmapsto u^D_{\phi, g^o, g^i}
\]
is linear. Consequently, continuity---and, indeed, analyticity---with respect to $(g^o,g^i)$ follows immediately. The main difficulty therefore lies in analyzing the dependence on the parameter $\phi$, which determines the shape of the interior boundary $\phi(\partial\Omega^i)$ of the perforated domain $\Omega^o \setminus \overline{\Omega^i[\phi]}$.

By exploiting a boundary integral equation formulation of problem \eqref{princeqpertu} and the Functional Analytic Approach developed by Lanza de Cristoforis, we prove in Theorem~\ref{thm:smoothrep} that a suitable restriction of the solution $u^D_{\phi,g^o,g^i}$ is $C^\infty$-smooth as a function of the triple $(\phi,g^o,g^i)$. In Theorem~\ref{thm:smoothrepbis}, {we strengthen} the result of Chapko, Kress, and Yoon \cite[Thm.~2.1]{ChKrYo98}---which establishes the existence of the {shape} derivative for a specific class of perturbations---by proving $C^\infty$-smoothness.

More precisely, we consider the Dirichlet-to-Neumann map $\Lambda$ defined by
\[
(\phi,g^o,g^i)\mapsto \Lambda[\phi,g^o,g^i]:=\frac{\partial}{\partial \nu_{\Omega^o}}u^D_{\phi,g^o,g^i},
\]
and prove that it depends smoothly on the variable $(\phi,g^o,g^i)$. We then focus on its dependence on the shape parameter $\phi$ and compute its partial differential with respect to $\phi$. Our main result is Theorem~\ref{thm:diffnormal}, where we derive an explicit formula for the differential of the normal derivative of the solution with respect to the infinite-dimensional shape parameter $\phi$.

{
In Section~\ref{s:neum}, we turn our attention to the shape sensitivity of the solution to a Neumann problem. For $\Omega^{o}$ and $\Omega^{i}$ satisfying condition \eqref{introsetconditions}, $\phi \in \mathcal{A}^{\Omega^{o}}_{\partial\Omega^{i}}$, and {Neumann boundary data}
\begin{equation}\label{introfunconditionsneu}
(f^o,f^i) \in {C_0^{\frac{\alpha}{2}; \alpha}([0,T] \times \partial\Omega^{o}) \times C_0^{\frac{\alpha}{2};\alpha}([0,T] \times \partial\Omega^{i})}\, ,
\end{equation}
we consider the Neumann counterpart of problem \eqref{princeqpertu}, namely the following boundary value problem for a function
\[
u \in C_{0}^{\frac{1+\alpha}{2}; 1+\alpha}\bigl([0,T] \times (\overline{\Omega^{o}} \setminus \Omega^{i}[\phi])\bigr):
\]

\begin{equation}\label{princeqpertuneu}
\begin{cases}
\partial_t u - \Delta u = 0 & \quad \text{in } ]0,T] \times (\Omega^{o} \setminus \overline{\Omega^{i}[\phi]}),
\\
\displaystyle \frac{\partial u}{\partial \nu_{\Omega^o}}(t,x) = f^o(t,x) & \quad \forall (t,x)\in [0,T] \times \partial \Omega^{o},
\\
\displaystyle \frac{\partial u}{\partial \nu_{\Omega^i[\phi]}}(t,x) = f^i(t,\phi^{-1}(x)) & \quad \forall (t,x)\in [0,T] \times \partial \Omega^{i}[\phi],
\\
u(0,\cdot)=0 & \quad \text{in } \overline{\Omega^{o}} \setminus \Omega^{i}[\phi].
\end{cases}
\end{equation}
We note that   $\nu_{\Omega^o}$ and $\nu_{\Omega^i[\phi]}$ denote respectively the outer unit normal fields to $\Omega^o$ and $\Omega^i[\phi]$.

By the results of \cite{DaLuMoMu25}, the unique solution
\[
u^N_{\phi,f^o,f^i}\in
C_{0}^{\frac{1+\alpha}{2}; 1+\alpha}\bigl([0,T] \times (\overline{\Omega^{o}} \setminus \Omega^{i}[\phi])\bigr)
\]
of problem \eqref{princeqpertuneu} depends smoothly on the triple $(\phi,f^o,f^i)$. We then compute the shape differential of the restriction of the solution to the exterior boundary, thereby extending the formula obtained by Chapko, Kress, and Yoon \cite{ChKrYo99}.}

We finally note that in problems \eqref{princeqpertu} and \eqref{princeqpertuneu}  the boundary condition depends on the diffeomorphism $\phi$. In particular this implies that the problems are sensitive to different parametrization of the same boundary, so that speaking of shape differential is a slight ambiguity.  Our approach also makes it possible to overcome this issue by considering $G^i,F^i:[0,T]\times \mathbb{R}^n \to \mathbb{R}$ and replacing the Dirichlet boundary condition on $[0,T] \times \partial \Omega^{i}[\phi]$ by:
\begin{equation*}
 u(t,x) = G^i(t,x)  \quad \forall (t,x)\in [0,T] \times \partial \Omega^{i}[\phi],
\end{equation*}
and the Neumann boundary condition by
\begin{equation*}
\frac{\partial u}{\partial \nu_{\Omega^i[\phi]}}(t,x) = F^i(t,x)  \quad \forall (t,x)\in [0,T] \times \partial \Omega^{i}[\phi].
\end{equation*}
However, this would require more regularity assumptions on the data $G^i,F^i$ and, since one of our aims is to develop a theory in the optimal H\"older setting, we preferred to consider the boundary conditions as stated in  \eqref{princeqpertu} and \eqref{princeqpertuneu}. Clearly, when one considers constant boundary conditions on the hole---so, in particular, the homogeneous ones as in Chapko, Kress and Yoon \cite{ChKrYo98,ChKrYo99}---the boundary conditions do not depend anymore on $\phi$ and the ambiguity ceases.

\subsection*{Structure of the paper}

The paper is organized as follows. Section~\ref{s:prel} contains preliminary material, mainly concerning potential theory for the heat equation. In Section~\ref{sec:pert}, we analyze the perturbed problem \eqref{princeqpertu} and prove the smooth dependence of its solution on the domain perturbation and the boundary data (see Theorem~\ref{thm:smoothrep}). In Section~\ref{s:inv}, we strengthen the result of Chapko, Kress, and Yoon \cite[Thm.~2.1]{ChKrYo98}, who established the existence of a {shape} derivative for a specific class of perturbations (see Theorem~\ref{thm:smoothrepbis}). {Section~\ref{s:diff} is devoted to the computation of shape differentials. In particular, in Theorem~\ref{thm:diffnormal} we derive an explicit formula for the shape differential of the normal derivative of the solution. In Section~\ref{ss:comparison}, we compare our result with that of Chapko, Kress, and Yoon \cite{ChKrYo98}. Section~\ref{s:neum} is devoted to the analysis of the Neumann problem and includes a comparison with the corresponding result of Chapko, Kress, and Yoon \cite{ChKrYo99}. Finally, in the appendix, we present an auxiliary result used in the proof of Theorem~\ref{cor:derneu}.}

\section{Preliminaries}\label{s:prel}

In this section, we recall some  basic notions{---mainly on potential theory---}for the heat equation. Indeed, our plan is to transform problem \eqref{princeqpertu} into a system of integral equations by exploiting the single layer potential for the heat equation.

We recall the definition of parabolic Schauder spaces. If $\alpha \in \mathopen]0,1[$, $T>0$ and {$\Omega$ is an open subset of $\mathbb{R}^n$}, then $C^{\frac{\alpha}{2};\alpha}([0, T] \times \overline{\Omega})$ denotes the space of
bounded continuous functions $u$ from $[0, T] \times \overline\Omega$ to $\mathbb{R}$ such that
\begin{align*}
    \|u\|_{C^{\frac{\alpha}{2};\alpha}([0, T] \times \overline{\Omega})} :=& \sup_{[0, T] \times \overline\Omega} |u| + \sup_{\substack{t_1,t_2 \in [0,T] \\ t_1 \neq t_2}} \,\sup_{x \in \overline\Omega} \frac{|u(t_1,x) - u(t_2,x)|}{|t_1-t_2|^\frac{\alpha}{2}}
    \\
    & + \sup_{t \in [0,T]} \,\sup_{\substack{x_1,x_2 \in \overline\Omega \\ x_1 \neq x_2}} \frac{|u(t,x_1) - u(t,x_2)|}{|x_1-x_2|^\alpha} < +\infty.
\end{align*}
Similarly, $C^{\frac{1+\alpha}{2}; 1+\alpha}([0, T] \times \overline{\Omega})$ denotes the space of
bounded continuous functions $u$ from $[0, T] \times \overline{\Omega}$ to $\mathbb{R}$ which are continuously differentiable
with respect to the space variables and such that
\begin{align*}
    \|u\|_{C^{\frac{1+\alpha}{2}; 1+\alpha}([0, T] \times \overline{\Omega})} :=& \sup_{[0, T] \times \overline{\Omega}} |u| + \sum_{i=1}^{n} \|\partial_{x_i} u\|_{C^{\frac{\alpha}{2}; \alpha}([0, T] \times \overline{\Omega})} 
    \\
    & + \sup_{\substack{t_1,t_2 \in [0,T] \\ t_1 \neq t_2}} \,\sup_{x \in \overline{\Omega}} \frac{|u(t_1,x) - u(t_2,x)|}{|t_1-t_2|^{\frac{1+\alpha}{2}}} < +\infty,
\end{align*}
{and  $C^{\frac{2+\alpha}{2}; 2+\alpha}([0, T] \times \overline{\Omega})$ denotes the space of
bounded continuous functions $u$ from $[0, T] \times \overline{\Omega}$ to $\mathbb{R}$ which are twice continuously differentiable
with respect to the space variables and one with respect to the time variable, and such that
\begin{align*}
    \|u\|_{C^{\frac{2+\alpha}{2}; 2+\alpha}([0, T] \times \overline{\Omega})} :=& \sup_{[0, T] \times \overline{\Omega}} |u| + \sum_{i=1}^{n} \|\partial_{x_i} u\|_{C^{\frac{1+\alpha}{2}; 1+\alpha}([0, T] \times \overline{\Omega})}+\|\partial_t u\|_{C^{\frac{\alpha}{2}; \alpha}([0, T] \times \overline{\Omega})} < +\infty.
\end{align*}
}
If $\Omega$ is an open subset of $\mathbb{R}^n$ of class {$C^{j,\alpha}$, $j\in\{0,1,2\}$,} we use the local parametrization of $\partial \Omega$ to define the space
{$C^{\frac{j+\alpha}{2}; j+\alpha}([0, T] \times \partial \Omega)$} in the natural way.

The subscript $0$ is used to denote a subspace consisting of functions that are
zero at $t = 0$, together with their time derivatives, if they exist. For example, for  $j \in \{0,1,2\}$
\[
\begin{split}
C_0^{\frac{j+\alpha}{2};j+\alpha}([0, T]& \times \overline{\Omega})\\&:= 
\Big\{ u \in {C^{\frac{j+\alpha}{2};j+\alpha}([0, T] \times \overline{\Omega})} \,:\,    \partial^k_t u (0,x)=  0 \quad \forall k=0,\ldots, \left[\frac{j+\alpha}{2}\right], \forall x \in \Omega\Big\}.
\end{split}
\]
Then $C_0^{\frac{j+\alpha}{2}; j+\alpha}([0, T] \times \partial \Omega)$  with $j \in \{0,1,2\}$
is similarly defined.

For functions in parabolic Schauder spaces, the partial derivative $D_x$ with respect to the space variable $x$ will be denoted by the gradient $\nabla$, while we will maintain the notation $\partial_t$ for the derivative with respect to the time variable $t$. 

For a comprehensive introduction to parabolic Schauder spaces we refer the reader to
classical monographs on the field, for example Lady\v{z}enskaja, Solonnikov, and Ural'ceva
\cite[Chapter 1]{LaSoUr68} (see also \cite{LaLu17,LaLu19}).

We recall now some well-known results on the {single layer potential} for the heat equation (for proofs
and detailed references, see Lady\v{z}enskaja, Solonnikov, and Ural'ceva
\cite[Chapter 4]{LaSoUr68}). 

To define the single  layer potential we need to introduce the function $S_{n} : \mathbb{R}^{1+n} \setminus
\{(0,0)\}\to \mathbb{R}$  by
\[ 
S_{n}(t,x):=
\left\{
\begin{array}{ll}
\frac{1}{(4\pi t)^{\frac{n}{2}} }e^{-\frac{|x|^{2}}{4t}}&{\mathrm{if}}\ (t,x)\in \mathopen]0,+\infty[ \times{\mathbb{R}}^{n}\,, 
\\
0 &{\mathrm{if}}\ (t,x)\in (\mathopen]-\infty,0]\times{\mathbb{R}}^{n})\setminus\{(0,0)\}.
\end{array}
\right.
\] 
It is well known that $S_n$ is a fundamental solution of the heat operator $\partial_t-\Delta$ in $\mathbb{R}^{1+n} \setminus \{(0,0)\}$.

Let $\alpha \in \mathopen]0,1[$ and $T>0$. Now let  $\Omega$  be an open bounded subset of $\mathbb{R}^n$ of class $C^{1,\alpha}$. For a density $\mu \in L^\infty\big([0,T] \times \partial\Omega\big)$, we define the single layer heat potential  as
\begin{equation*} 
    v_{\Omega} [\mu](t,x) := \int_{0}^{t} \int_{\partial \Omega} S_{n}(t-\tau,x-y) \mu(\tau, y)\,d\sigma_y d\tau \quad \forall\,(t,x) \in [0, T] \times \mathbb{R}^n\, ,
\end{equation*}
and we set
\begin{equation*}
    V_{\partial\Omega}[\mu] := v_{\Omega}[\mu]_{|[0,T]\times \partial\Omega} 
\end{equation*} 
for denoting the trace of the single layer heat potential on $[0,T]\times \partial\Omega$. In order to describe the normal derivative of the single layer potential, we also define
\begin{equation*}
\begin{split}
    W^*_{\partial \Omega}[\mu](t,x) := \int_{0}^t\int_{\partial\Omega} 
    \frac{\partial}{\partial \nu_\Omega(x)} S_{n}(t-\tau,x-y) \mu(\tau,y)\,d\sigma_yd\tau 
     \qquad \forall\,(t,x) \in [0,T] \times \partial\Omega\,,
\end{split}
\end{equation*}
{where $\nu_\Omega$ denotes {the} outer unit normal to $\Omega$.}
We collect some properties of the single layer potential in the following theorem (for a proof and detailed references see \cite[Thm.~1]{DaLuMoMu25}). Here, we just mention that it follows from results from \cite{DaLu23,LaLu17,LaLu19} and from  Baderko \cite[Thm.~3.4]{Ba97}, Brown \cite{Br89,Br90}, and Costabel \cite{Co90}. 

\begin{theorem}\label{thmsl}
Let $\alpha \in \mathopen]0,1[$ and $T>0$. { Let $j \in \{0,1\}.$} Let $\Omega$ be a bounded open subset of $\mathbb{R}^n$ of class $C^{{j+1},\alpha}$.  Then the following statements hold. 
\begin{itemize}

\item[(i)] Let $\mu \in L^\infty([0,T] \times \partial\Omega)$. Then the function $v_{\Omega}[\mu]$ is continuous and 
$v_{\Omega}[\mu] \in C^\infty(]0,T[ \times (\mathbb{R}^n \setminus \partial\Omega))$. 
Moreover $v_{\Omega}[\mu]$ solves the heat equation 
in $]0,T]\times (\mathbb{R}^n \setminus \partial\Omega)$.

\item[(ii)] Let $v_{\Omega}^{+}[\mu]$ and $v_{\Omega}^-[\mu]$ denote 
the restrictions of $v_\Omega[\mu]$ to $[0,T] \times \overline{\Omega}$ and to $[0,T]\times \overline{\Omega^{-}}$, respectively. Then, the map from  $C_0^{\frac{{j+}\alpha}{2}; {j+} \alpha}([0,T] \times \partial\Omega)$ to  $C_{0}^{\frac{{j+}1+\alpha}{2}; {j+}1+\alpha}([0,T] \times \overline{\Omega})$ that takes $\mu$ to $v_{\Omega}^+[\mu]$ is linear and continuous. If $R>0$ is such that $\overline{\Omega}$ is contained in the ball $B(0,R)$ of center $0$ and radius $R$, then the map from  $C_0^{\frac{{j+}\alpha}{2};  {j+}\alpha}([0,T] \times \partial\Omega)$ to  $C_{0}^{\frac{{j+}1+\alpha}{2}; {j+}1+\alpha}([0,T] \times (\overline{B(0,R)}\setminus \Omega))$ that takes $\mu$ to $v_{\Omega}^-[\mu]_{|[0,T] \times (\overline{B(0,R)}\setminus \Omega)}$ is also linear and continuous.

\item[(iii)] Let $\mu \in C_0^{\frac{\alpha}{2};  \alpha}([0,T] \times \partial\Omega)$. Then the following jump relations hold: 
\begin{align*}
&\frac{\partial}{\partial \nu_\Omega}v_{\Omega}^\pm[\mu](t,x)  = \pm \frac{1}{2}\mu(t,x)
 +W_{\partial\Omega}^*[\mu](t,x),\\
 &\nabla v_{\Omega}^\pm[\mu](t,x)  = \pm \frac{1}{2}\mu(t,x)\nu_\Omega (x)
 + \int_{0}^t\int_{\partial\Omega} 
    \nabla S_{n}(t-\tau,x-y) \mu(\tau,y)\,d\sigma_yd\tau,
 \end{align*}
 for all $(t,x) \in [0,T] \times \partial\Omega$.
\item[(iv)] The operator $V_{\partial \Omega}$ is an isomorphism from the space $C_0^{\frac{{j+}\alpha}{2}; {j+}\alpha}([0,T] \times \partial\Omega)$ to $C_0^{\frac{{j+}1+\alpha}{2}; {j+}1+\alpha}([0,T] \times \partial\Omega)$.
\item[(v)] The operator $W_{\partial \Omega}^*$ is  compact from $C_0^{\frac{{j+}\alpha}{2};{j+}  \alpha}([0,T] \times \partial\Omega)$ to itself.
\end{itemize}
\end{theorem}

\section{The perturbed Dirichlet problem (\ref{princeqpertu})}\label{sec:pert}

We now consider the perturbed Dirichlet problem \eqref{princeqpertu}.  The first step in our analysis consists in transforming   problem \eqref{princeqpertu} into a  system of integral equations and we do so by exploiting a representation formula of the solution in terms of the single layer potential.

Let $\alpha \in \mathopen]0,1[$, $T>0$ and let $\Omega^{o}$, $\Omega^{i}$ be as in \eqref{introsetconditions}.  {For the sake of brevity we set}
\[
\mathcal{X}_0^{\frac{\alpha}{2};  \alpha}:=C_0^{\frac{\alpha}{2};  \alpha}([0,T] \times \partial\Omega^{o}) \times C_0^{\frac{\alpha}{2};  \alpha}([0,T] \times \partial\Omega^{i})
\]
and
\[
\mathcal{X}_0^{\frac{1+\alpha}{2};  1+\alpha}:=C_0^{\frac{1+\alpha}{2};  1+\alpha}([0,T] \times \partial\Omega^{o}) \times C_0^{\frac{1+\alpha}{2};  1+\alpha}([0,T] \times \partial\Omega^{i})\, .
\]

Then we define the map $\mathcal{M}=(\mathcal{M}_1,\mathcal{M}_2)$  from $\mathcal{A}^{\Omega^{o}}_{\partial\Omega^{i}}$ to $\mathcal{L}(\mathcal{X}_0^{\frac{\alpha}{2};  \alpha},\mathcal{X}_0^{\frac{1+\alpha}{2};  1+\alpha})$ that takes $\phi \in \mathcal{A}^{\Omega^{o}}_{\partial\Omega^{i}}$ to the linear and continuous operator from $\mathcal{X}_0^{\frac{\alpha}{2};  \alpha}$ to $\mathcal{X}_0^{\frac{1+\alpha}{2};  1+\alpha}$ given by
\begin{equation}\label{M}
\begin{aligned}
\mathcal{M}_1[\phi](\mu^o,\mu^i) & := V_{\partial\Omega^{o}}  [\mu^o] + v^-_{\Omega^{i}[\phi]}[\mu^i \circ (\phi^T)^{(-1)}]_{ |[0,T] \times \partial\Omega^{o}} \qquad &{\mbox{ on } \, [0,T] \times \partial\Omega^{o}},
\\
\mathcal{M}_2[\phi](\mu^o,\mu^i) & := v^+_{\Omega^{o}}[\mu^o]_{|[0,T] \times \partial\Omega^{i}[\phi]}\circ \phi^T + V_{\partial\Omega^{i}[\phi]}[\mu^i\circ (\phi^T)^{(-1)}] \circ \phi^T \,  &{\mbox{ on } \, [0,T] \times \partial\Omega^{i}},
\end{aligned}
\end{equation}
for all $(\mu^o,\mu^i) \in \mathcal{X}_0^{\frac{\alpha}{2};  \alpha}$. Here {above} and in the sequel, if $A$ is a subset of $\mathbb{R}^n$,
$T >0$ and $h$ is a map from $A$ to $\mathbb{R}^n$, we denote by $h^T$ the map from  $[0,T] \times A$
 to  $[0,T] \times \mathbb{R}^n$ defined by 
\[
h^T(t,x) := (t, h(x)) \quad \forall (t,x) \in  [0,T] \times A.
\]
From the definition of $\mathcal{M}$ and the properties of the single layer potential, we  deduce the validity of the following:
\begin{proposition}\label{prop M=0}
Let $\alpha \in \mathopen]0,1[$ and $T>0$. Let $\Omega^{o}$, $\Omega^{i}$ be as in \eqref{introsetconditions}. {Let
\begin{equation*}
\phi\in \mathcal{A}^{\Omega^{o}}_{\partial\Omega^{i}}\, ,\qquad (\mu^o,\mu^i) \in \mathcal{X}_0^{\frac{\alpha}{2};  \alpha} \, , \qquad (g^o, g^i ) \in  \mathcal{X}^{\frac{1+\alpha}{2};1+  \alpha}_0\, .
\end{equation*}}
{Then the function 
\[
(v^+_{\Omega^{o}} [\mu^o] + v^-_{\Omega^{i}[\phi]}[\mu^i\circ (\phi^T)^{(-1)}])_{| [0,T] \times (\overline{\Omega^{o}} \setminus {\Omega^{i}[\phi]})}
\]}  
is a solution of problem \eqref{princeqpertu} if and only if 
\begin{equation}\label{M=0}
\mathcal{M}[\phi](\mu^o,\mu^i)=(g^o, g^i).
\end{equation}
Moreover, for each {$\phi \in \mathcal{A}^{\Omega^{o}}_{\partial\Omega^{i}}$ and $(g^o, g^i ) \in  {\mathcal{X}^{\frac{1+\alpha}{2};  1+\alpha}_0}$,} there exists a unique pair $(\mu^o,\mu^i)\in  {\mathcal{X}^{\frac{\alpha}{2};  \alpha}_0}$ such that \eqref{M=0} holds true. We denote such a pair by $(\mu^o[\phi, g^o, g^i],\mu^i[\phi, g^o, g^i])$.
\end{proposition}
\begin{proof}
We first note that if $\phi \in \mathcal{A}^{\Omega^{o}}_{\partial\Omega^{i}}$, {then} $(\mu^o, \mu^i) \in C_0^{\frac{\alpha}{2};  \alpha}([0,T] \times \partial\Omega^{o}) \times C_0^{\frac{\alpha}{2};  \alpha}([0,T] \times \partial\Omega^{i})$ if and only if
    \begin{equation*}(\mu^o,\mu^i \circ (\phi^T)^{(-1)}) \in C_0^{\frac{\alpha}{2};  \alpha}([0,T] \times \partial\Omega^{o}) \times C_0^{\frac{\alpha}{2};  \alpha}([0,T] \times \phi(\partial\Omega^{i})).
    \end{equation*}
 Then, by the properties of the single layer potential (cf.~Theorem \ref{thmsl}), by a change of variable on $\phi(\partial \Omega^{i})$,  and by the definition \eqref{M} of $\mathcal{M}$, we deduce that the function
    \begin{equation*}
(v^+_{\Omega^{o}} [{\mu^o}] + v^-_{\Omega^{i}[\phi]}[{\mu^i} \circ (\phi^T)^{(-1)}])_{|[0,T]\times \overline{\Omega^{o}} \setminus \Omega^{i}[\phi]}
    \end{equation*}
    is the unique solution of problem \eqref{princeqpertu} if and only if \eqref{M=0} is satisfied. To conclude,  by the invertibility of the trace of the single layer potential (see Theorem \ref{thmsl}) and by the definition of $\mathcal{M}$, we deduce 
    that for each $\phi \in \mathcal{A}^{\Omega^{o}}_{\partial\Omega^{i}}$ the map $\mathcal{M}[\phi]$ is a linear and continuous invertible operator in $\mathcal{L}(\mathcal{X}_0^{\frac{\alpha}{2};  \alpha},\mathcal{X}_0^{\frac{1+\alpha}{2};  1+\alpha})$. Indeed, up to a push-forward of the second density $\mu^i$ to $[0,T] \times \partial\Omega^i[\phi]$ and a pull-back of the second component $\mathcal{M}_2[\phi]$ to $[0,T] \times \partial\Omega^{i}$, the operator $\mathcal{M}[\phi]$ is exactly the trace of a single layer potential on the boundary of $\Omega^o\setminus \overline{\Omega^i[\phi]}$.
     \end{proof}

     \begin{remark}
     By Proposition \ref{prop M=0}, we clearly have
     \[
     u^D_{\phi,g^o,g^i}=(v^+_{\Omega^{o}} [\mu^o[\phi,g^o,g^i]] + v^-_{\Omega^{i}[\phi]}[\mu^i[\phi,g^o,g^i]\circ (\phi^T)^{(-1)}])_{| [0,T] \times (\overline{\Omega^{o}} \setminus {\Omega^{i}[\phi]})}
     \]
     for all {$\phi \in \mathcal{A}^{\Omega^{o}}_{\partial\Omega^{i}}$ and $(g^o, g^i ) \in  {\mathcal{X}^{\frac{1+\alpha}{2};  1+\alpha}_0}$,} where we recall that $u^D_{\phi,g^o,g^i}$ is the unique solution to problem \eqref{princeqpertu}. 
     \end{remark}

In order to study the regularity of the operator $\mathcal{M}$, we  need some results on the operators involved in its definition. We begin with the following technical lemma of Lanza de Cristoforis and Rossi \cite[p.~166]{LaRo04} and Lanza de Cristoforis \cite[Prop.~1]{La07} on the smoothness of certain maps related to the change of variables in integrals and to the {pull-back} of the
outer normal field.

\begin{lemma}\label{lemma change of variable}
	Let $\alpha \in \mathopen]0,1[$. Let $\Omega^{i}$ be an open bounded connected subset of $\mathbb{R}^n$ of class $C^{1,\alpha}$ with connected exterior $\mathbb{R}^n\setminus \overline{\Omega^{i}}$. Then the following hold.
    \begin{itemize}
        \item[(i)] For each $\phi \in \mathcal{A}_{\partial\Omega^{i}}$ there exists a unique $\tilde{\sigma}_n[\phi] \in C^{0,\alpha}(\partial\Omega^{i})$ such that 
        \[
    	\int_{\phi(\partial\Omega^{i})} f(y) \,d\sigma_y = \int_{\partial\Omega^{i}} f(\phi(s)) \, \tilde{\sigma}_n[\phi](s) \,d\sigma_s \quad\forall f \in L^1(\phi(\partial\Omega^{i})).
    	\]
     Moreover, $\tilde{\sigma}_n[\phi] >0$ and the map taking $\phi$ to $\tilde{\sigma}_n[\phi] $ is real analytic from $\mathcal{A}_{\partial\Omega^{i}}$ to $C^{0,\alpha}(\partial\Omega^{i})$.

     \item[(ii)] The map from  $\mathcal{A}_{\partial\Omega^{i}}$ to $C^{0,\alpha}(\partial\Omega^{i})$  that  takes $\phi$ to ${\nu_{\Omega^{i}[\phi]}} \circ \phi$ is real analytic.
    \end{itemize}
\end{lemma}

Similarly, we collect in the lemma below, some smoothness results concerning the operators involved in the definition of $\mathcal{M}$.

\begin{lemma}\label{lem:reg}
Let $\alpha \in \mathopen]0,1[$ and $T>0$. Let $\Omega^{o}$, $\Omega^{i}$ be as in \eqref{introsetconditions}. Then the following statements hold.
\begin{itemize}
\item[(i)] The map that takes $\phi$ to the  operator
\[
\mu^i\mapsto V_{\partial\Omega^{i}[\phi]}[\mu^i\circ (\phi^T)^{(-1)}] \circ \phi^T
\]
is of class $C^\infty$ from $\mathcal{A}^{\Omega^{o}}_{\partial\Omega^{i}}$ to $\mathcal{L}(C_0^{\frac{\alpha}{2};  \alpha}([0,T] \times \partial\Omega^{i}),C_0^{\frac{1+\alpha}{2};  1+\alpha}([0,T] \times \partial\Omega^{i}))$. 
\item[(ii)] The map that takes $\phi$ to the operator
\[
\mu^i\mapsto  v^-_{\Omega^{i}[\phi]}[\mu^i \circ (\phi^T)^{(-1)}]_{|[0,T] \times\partial \Omega^{o}}
\]
is of class $C^\infty$
from $\mathcal{A}^{\Omega^{o}}_{\partial\Omega^{i}}$ to $\mathcal{L}(C_0^{\frac{\alpha}{2};  \alpha}([0,T] \times \partial\Omega^{i}),C_0^{\frac{1+\alpha}{2};  1+\alpha}([0,T] \times \partial\Omega^{o}))$.
\item[(iii)] The map that takes $\phi$ to the operator
\[
\mu^o\mapsto v^+_{\Omega^{o}}[\mu^o]_{|[0,T] \times \partial\Omega^{i}[\phi]}\circ \phi^T 
\]
is of class $C^\infty$
from $\mathcal{A}^{\Omega^{o}}_{\partial\Omega^{i}}$ to $\mathcal{L}(C_0^{\frac{\alpha}{2};  \alpha}([0,T] \times \partial\Omega^{o}),C_0^{\frac{1+\alpha}{2};  1+\alpha}([0,T] \times \partial\Omega^{i}))$.
\end{itemize}
\end{lemma}
\begin{proof}
Statement (i) follows by \cite[Thm.~5.4]{DaLu23}. We now consider statement (ii). We first note that
\[
v^-_{\Omega^{i}[\phi]}[\mu^i \circ (\phi^T)^{(-1)}](t,x)= \int_{0}^{t} \int_{\partial \Omega^{i}} S_{n}(t-\tau,x-\phi(s)) \mu^i(\tau, s) \, \tilde{\sigma}_n[\phi](s) \,d\sigma_s d\tau \quad \forall\,(t,x) \in [0, T] \times \partial \Omega^{o}\, ,
\]
for all $(\phi,\mu^i)\in \mathcal{A}^{\Omega^{o}}_{\partial\Omega^{i}} \times C_0^{\frac{\alpha}{2};  \alpha}([0,T] \times \partial\Omega^{i})$.  Since $\phi \in \mathcal{A}^{\Omega^{o}}_{\partial\Omega^{i}}$, then
\[
x-\phi(s) \neq 0 \quad\forall (x,s) \in  \partial \Omega^{o}  \times \partial\Omega^{i}.
\]
 Accordingly, \cite[Lemmas  A.2, A.3]{DaLu23} on the regularity of time-dependent integral operators with non-singular kernels and of superposition operators and  Lemma \ref{lemma change of variable} imply that the map that takes $(\phi,\mu^i)\in \mathcal{A}^{\Omega^{o}}_{\partial\Omega^{i}} \times C_0^{\frac{\alpha}{2};  \alpha}([0,T] \times \partial\Omega^{i})$ to $v^-_{\Omega^{i}[\phi]}[\mu^i \circ (\phi^T)^{(-1)}]_{|[0,T]\times \partial \Omega^{o}} \in C_0^{\frac{1+\alpha}{2};  1+\alpha}([0,T] \times \partial\Omega^{o})$ is of class $C^\infty$. Since for $\phi \in \mathcal{A}^{\Omega^{o}}_{\partial\Omega^{i}}$ the map
 \[
 \mu^i \mapsto v^-_{\Omega^{i}[\phi]}[\mu^i \circ (\phi^T)^{(-1)}]_{|[0,T] \times\partial \Omega^{o}}
 \]
 is linear from $C_0^{\frac{\alpha}{2};  \alpha}([0,T] \times \partial\Omega^{i})$ to $C_0^{\frac{1+\alpha}{2};  1+\alpha}([0,T] \times \partial\Omega^{o})$, by standard calculus in Banach spaces, we deduce the validity of (ii). Finally, the proof of (iii) is similar to the one of statement (ii) and is accordingly left to the reader.
\end{proof}

By Theorem \ref{thmsl} and Lemma \ref{lem:reg}, we can now deduce the validity of the following result on the smoothness of the map $\mathcal{M}$.

\begin{proposition}\label{prop Mrealanal}
  Let $\alpha \in \mathopen]0,1[$ and $T>0$. Let $\Omega^{o}$, $\Omega^{i}$ be as in \eqref{introsetconditions}. Then the map $\mathcal{M}$ from $\mathcal{A}^{\Omega^{o}}_{\partial\Omega^{i}}$ to $\mathcal{L}(\mathcal{X}_0^{\frac{\alpha}{2};  \alpha},\mathcal{X}_0^{\frac{1+\alpha}{2};  1+\alpha})$ is of class $C^\infty$.
\end{proposition}

By  Proposition \ref{prop Mrealanal} and   the invertibility of $\mathcal{M}$ we can prove that the pair $(\mu^o[\phi, g^o, g^i],\mu^i[\phi, g^o, g^i])$ is of class $C^\infty$ as a function of the triple $(\phi, g^o, g^i )$. We do so in the following theorem.

\begin{theorem}\label{Lambda Thm}
  Let $\alpha \in \mathopen]0,1[$ and $T>0$. Let $\Omega^{o}$, $\Omega^{i}$ be as in \eqref{introsetconditions}. Then the map from $\mathcal{A}^{\Omega^{o}}_{\partial\Omega^{i}}\times \mathcal{X}_0^{\frac{1+\alpha}{2};  1+\alpha}$ to $\mathcal{X}_0^{\frac{\alpha}{2};  \alpha}$ that takes $(\phi, g^o, g^i )$ to $(\mu^o[\phi, g^o, g^i],\mu^i[\phi, g^o, g^i])$ is of class $C^\infty$.\end{theorem}
\begin{proof}
Let $(\phi, g^o, g^i ) \in \mathcal{A}^{\Omega^{o}}_{\partial\Omega^{i}}\times \mathcal{X}_0^{\frac{1+\alpha}{2};  1+\alpha}$. By the invertibility of $\mathcal{M}[\phi]$, we have
\[
(\mu^o[\phi, g^o, g^i],\mu^i[\phi, g^o, g^i])=(\mathcal{M}[\phi])^{(-1)}(g^o, g^i )\, .
\]
Then we recall that the map that takes a linear invertible operator to its inverse is real analytic (cf.~Hille and Phillips \cite[Thms. 4.3.2 and 4.3.4]{HiPh57}), and therefore of class $C^\infty$.
Accordingly, by Proposition \ref{prop Mrealanal}, we also deduce that  the map from $\mathcal{A}^{\Omega^{o}}_{\partial\Omega^{i}}$ to $\mathcal{L}(\mathcal{X}_0^{\frac{1+\alpha}{2};  1+\alpha}, \mathcal{X}_0^{\frac{\alpha}{2};  \alpha})$ that takes $\phi$ to $\mathcal{M}[\phi]^{(-1)}$ is of class $C^\infty$. Since the bilinear map from $\mathcal{L}(\mathcal{X}_0^{\frac{1+\alpha}{2};  1+\alpha}, \mathcal{X}_0^{\frac{\alpha}{2};  \alpha})\times \mathcal{X}_0^{\frac{1+\alpha}{2};  1+\alpha}$ to $\mathcal{X}_0^{\frac{\alpha}{2};  \alpha}$ that takes a pair $(G,u)$ to $G(u)$ is continuous, we deduce that 
\[
(\phi, g^o, g^i)\mapsto(\mathcal{M}[\phi])^{(-1)}(g^o, g^i )
\]
is of class $C^{\infty}$, and thus the proof is complete.
\end{proof}

By Theorem \ref{Lambda Thm} we know that the densities  $\mu^o[\phi, g^o, g^i]$ and $\mu^i[\phi, g^o, g^i]$ are smooth functions of the shape parameter $\phi$ and of the Dirichlet data $g^o$ and $g^i$. As a consequence by plugging them in the integral representation of the solution, we can deduce a smoothness result for suitable restrictions of the solution $u^D_{\phi,g^o,g^i}$. Therefore, in the following theorem, {we take a bounded open set $\Omega_\mathtt{int}$  such that
\[
{\overline \Omega_\mathtt{int}}\subseteq  \Omega^{o}
\] 
 and we consider the function $u^D_{\phi,g^o,g^i}$ for those diffeomorphisms $\phi$ such that $\overline{\Omega_\mathtt{int}}\subseteq \Omega^{o} \setminus \overline{\Omega^{i}[\phi]}$.

\begin{theorem}\label{thm:smoothrep}
  Let $\alpha \in \mathopen]0,1[$ and $T>0$. Let $\Omega^{o}$, $\Omega^{i}$ be as in \eqref{introsetconditions}.  Let $\Omega_\mathtt{int}$ be a bounded open subset of {$\Omega^{o}$ such that   $\overline{\Omega_\mathtt{int}}\subseteq  \Omega^{o}$} 
and let $Q_\mathtt{int}$ be  the open subset of $\mathcal{A}^{\Omega^{o}}_{\partial\Omega^{i}}$ defined by   
\begin{equation}\label{cond Omega int}
Q_\mathtt{int} := \Big\{ \phi \in \mathcal{A}^{\Omega^{o}}_{\partial\Omega^{i}} \colon   \overline{\Omega_\mathtt{int}} \subseteq \Omega^{o} \setminus \overline{\Omega^{i}[\phi]} \Big \}\, .
\end{equation}
Then the map from $Q_\mathtt{int} \times \mathcal{X}_{0}^{\frac{1+\alpha}{2}; 1+\alpha}$ to $C_{0}^{\frac{1+\alpha}{2}; 1+\alpha}([0,T] \times \overline{\Omega_\mathtt{int}})$ that takes $(\phi,g^o,g^i)$ to $\left(u^D_{\phi,g^o,g^i}\right)_{|[0,T]\times\overline{\Omega_\mathtt{int}}}$ is of class $C^\infty$.
\end{theorem}

\begin{proof}
We assume, without loss of generality, that $\Omega_\mathtt{int}$ is of class $C^{1,\alpha}$.  Indeed, since smoothness is a local property, for $\phi \in Q_\mathtt{int}$ we can replace  $\Omega_\mathtt{int}$ by a $C^{1,\alpha}$ set $\tilde \Omega_\mathtt{int} \supseteq \Omega_\mathtt{int}$ such that $\overline{\tilde \Omega_\mathtt{int}} \subseteq \Omega^{o} \setminus \overline{\Omega^{i}[\phi]}$  and prove the smoothness around $\phi$. By Proposition \ref{prop M=0},  for every $\phi \in Q_\mathtt{int}$ we have 
    \begin{equation}\label{eq:repa}
    \begin{split}
        u^D_{\phi,g^o,g^i}(t,x) &=\int_{0}^{t} \int_{\partial \Omega^{o}} S_{n}(t-\tau, x-y) \mu^o[\phi,g^o,g^i](\tau, y) \,d\sigma_y d\tau 
        \\ 
      &  \qquad +
        \int_{0}^{t} \int_{\phi(\partial \Omega^{i})} S_{n}(t-\tau, x-\tilde{y}) \mu^i[\phi,g^o,g^i] \circ (\phi^T)^{(-1)} (\tau,\tilde{y})\,d\sigma_{\tilde{y}} d\tau
        \\
        &= \int_{0}^{t} \int_{\partial \Omega^{o}} S_{n}(t-\tau, x-y) \mu^o[\phi,g^o,g^i](\tau, y) \,d\sigma_y d\tau 
        \\
       & \qquad  +\int_{0}^{t} \int_{\partial \Omega^{i}} S_{n}(t-\tau, x-\phi(y)) \mu^i[\phi,g^o,g^i](\tau, y) \tilde{\sigma}_n[\phi](y) \,d\sigma_y d\tau \, ,
    \end{split}
    \end{equation}
    for all $(t,x) \in [0,T] \times \overline{\Omega_\mathtt{int}}$. By \eqref{cond Omega int} we have
    \begin{equation*}
        x-y \neq 0 \quad\forall (x,y) \in  \overline{\Omega_\mathtt{int}} \times \partial \Omega^{o} \quad \text{and} \quad x-\phi(y) \neq 0 \quad\forall (x,y) \in  \overline{\Omega_\mathtt{int}} \times \partial\Omega^{i}.
    \end{equation*}
 Accordingly, \cite[Lemmas  A.2, A.3]{DaLu23} on the regularity of time-dependent integral operators with non-singular kernels and of superposition operators,  Theorem \ref{Lambda Thm}, and  Lemma \ref{lemma change of variable} imply that the integral operators in {the right hand side of} \eqref{eq:repa} define $C^\infty$ maps from $Q_\mathtt{int} \times \mathcal{X}_{0}^{\frac{1+\alpha}{2}; 1+\alpha}$ to $C_{0}^{\frac{1+\alpha}{2}; 1+\alpha}([0,T] \times \overline{\Omega_\mathtt{int}})$.
 \end{proof}

}

\section{An application to a map in inverse problems} \label{s:inv}

{In Chapko, Kress, and Yoon \cite{ChKrYo98} the authors considered the normal derivative of the solution to a specific {Dirichlet}   problem in $[0,T]\times \mathbb{R}^2$ and  established the existence of the shape derivative with respect to the interior boundary {curve.} The problem they considered coincides with \eqref{princeqpertu} if we take $g^i=0$. This result was then used to recover the interior boundary curve of an arbitrary-shaped annulus from overdetermined Cauchy data on the exterior boundary curve.

Inspired by \cite{ChKrYo98}, {we show} the $C^\infty$ smoothness of the same shape map they considered, but in the case of the more general problem  \eqref{princeqpertu}. We also provide a characterization of the shape {differential} of such a map. Namely, from Theorem \ref{thm:smoothrep} we  deduce the following.}

\begin{theorem}\label{thm:smoothrepbis}
  Let $\alpha \in \mathopen]0,1[$ and $T>0$. Let $\Omega^{o}$, $\Omega^{i}$ be as in \eqref{introsetconditions}. Then the map from $\mathcal{A}^{\Omega^{o}}_{\partial\Omega^{i}} \times {\mathcal{X}^{\frac{1+\alpha}{2};1+  \alpha}_0}$ to $C_{0}^{\frac{\alpha}{2}; \alpha}([0,T] \times \partial \Omega^{o})$ that takes $(\phi,g^o,g^i)$ to $\frac{\partial}{\partial \nu_{\Omega^{o}}}u^D_{\phi,g^o,g^i}$ is of class $C^\infty$.
\end{theorem}

\begin{proof}
By Proposition \ref{prop M=0},  for every $(\phi, g^o,g^i) \in \mathcal{A}^{\Omega^{o}}_{\partial\Omega^{i}}\times {\mathcal{X}^{\frac{1+\alpha}{2};1+  \alpha}_0}$ we have 
\[
    \begin{split}
        u^D_{\phi,g^o,g^i}(t,x) 
        &= \int_{0}^{t} \int_{\partial \Omega^{o}} S_{n}(t-\tau, x-y) \mu^o[\phi,g^o,g^i](\tau, y) \,d\sigma_y d\tau 
        \\
       & \qquad  +\int_{0}^{t} \int_{\partial \Omega^{i}} S_{n}(t-\tau, x-\phi(y)) \mu^i[\phi,g^o,g^i](\tau, y) \tilde{\sigma}_n[\phi](y) \,d\sigma_y d\tau \, ,
    \end{split}
\]
    for all $(t,x) \in [0,T] \times (\overline{\Omega^{o}} \setminus \Omega^{i}[\phi])$, where $\mu^o[\cdot]$ and $\mu^i[\cdot]$ are as in Proposition \ref{prop M=0}. As a consequence, by Theorem \ref{thmsl},
\[
    \begin{split}
       \frac{\partial}{\partial \nu_{\Omega^{o}}} u^D_{\phi,g^o,g^i}(t,x) 
        &=\frac{1}{2}\mu^o[\phi,g^o,g^i](t, x)+  \int_{0}^{t} \int_{\partial \Omega^{o}}  \frac{\partial}{\partial \nu_{\Omega^o}(x)} S_{n}(t-\tau, x-y) \mu^o[\phi,g^o,g^i](\tau, y) \,d\sigma_y d\tau 
        \\
       & \qquad  +\int_{0}^{t} \int_{\partial \Omega^{i}} \frac{\partial}{\partial \nu_{\Omega^o}(x)}  S_{n}(t-\tau, x-\phi(y)) \mu^i[\phi,g^o,g^i](\tau, y) \tilde{\sigma}_n[\phi](y) \,d\sigma_y d\tau \, ,
    \end{split}
\]
    for all $(t,x) \in [0,T] \times  \partial \Omega^{o}$. Again, by Theorems \ref{thmsl} and \ref{Lambda Thm}, we deduce that the map  from $\mathcal{A}^{\Omega^{o}}_{\partial\Omega^{i}} \times {\mathcal{X}^{\frac{1+\alpha}{2};1+  \alpha}_0}$ to $C_{0}^{\frac{\alpha}{2}; \alpha}([0,T] \times \partial {\Omega^o})$ that takes $(\phi,g^o,g^i)$ to the function
    \[
    \frac{1}{2}\mu^o[\phi,g^o,g^i](t, x)+  \int_{0}^{t} \int_{\partial \Omega^{o}}  \frac{\partial}{\partial \nu_{\Omega^o}(x)} S_{n}(t-\tau, x-y) \mu^o[\phi,g^o,g^i](\tau, y) \,d\sigma_y d\tau 
    \]
    of the variable $(t,x) \in [0,T] \times  \partial \Omega^{o}$ is of class $C^\infty$. Moreover, by arguing as in the proof of Theorem \ref{thm:smoothrep}, we deduce that the map  from $\mathcal{A}^{\Omega^o}_{\partial\Omega^{i}} \times {\mathcal{X}^{\frac{1+\alpha}{2};1+  \alpha}_0}$ to $C_{0}^{\frac{\alpha}{2}; \alpha}([0,T] \times \partial \Omega^{o})$ that takes $(\phi,g^o,g^i)$ to the function
    \[
\int_{0}^{t} \int_{\partial \Omega^{i}} \frac{\partial}{\partial \nu_{\Omega^o}(x)}  S_{n}(t-\tau, x-\phi(y)) \mu^i[\phi,g^o,g^i](\tau, y) \tilde{\sigma}_n[\phi](y) \,d\sigma_y d\tau
    \]
    of the variable $(t,x) \in [0,T] \times  \partial \Omega^{o}$ is of class $C^\infty$. Thus the proof is complete.
     \end{proof}

\subsection{On the computation of {the shape differential}}  \label{s:diff}

The next step is to compute the shape differential of the map $\phi \mapsto \frac{\partial}{\partial \nu_{\Omega^{o}}}u^D_{\phi,g^o,g^i}$.  As we have mentioned in the Introduction, such differential can be seen as the partial differential with respect to $\phi$ of the Dirichlet-to-Neumann-type map $\Lambda$ defined by
\[
(\phi,g^o, g^i)\mapsto \Lambda[\phi,g^o, g^i]:=\frac{\partial}{\partial \nu_{\Omega^o}}u^D_{\phi,g^o, g^i}\, .
\]
From now on,  we consider problem \eqref{princeqpertu} with $g^o$ and $g^i$ fixed and we focus on the dependence on the infinite dimensional shape parameter $\phi$. As a consequence, we simply denote by {$u^D_{\phi}$} the solution of problem \eqref{princeqpertu} and by    $(\mu^o[\phi],\mu^i[\phi])$  the solution to equation \eqref{M=0}.

We start by computing the differentials of some functionals related to layer potential operators involved in the definition of $\mathcal{M}$.

  Let $\alpha \in \mathopen]0,1[$ and $T>0$. Let $\Omega^{o}$, $\Omega^{i}$ be as in \eqref{introsetconditions}.

\begin{itemize}
\item We denote by $V_1$ the $C^\infty$ map from $\mathcal{A}^{\Omega^{o}}_{\partial\Omega^{i}}\times C_0^{\frac{\alpha}{2};  \alpha}([0,T] \times \partial\Omega^{i})$ to $C_0^{\frac{1+\alpha}{2};  1+\alpha}([0,T] \times \partial\Omega^{i})$ that takes $(\phi,\mu) \in \mathcal{A}^{\Omega^{o}}_{\partial\Omega^{i}}\times C_0^{\frac{\alpha}{2};  \alpha}([0,T] \times \partial\Omega^{i})$ to 
\begin{equation}\label{eq:V1}
V_1[\phi,\mu]:= V_{\partial\Omega^{i}[\phi]}[\mu \circ (\phi^T)^{(-1)}] \circ \phi^T \, .
\end{equation}
\item We denote by $V_2$ the $C^\infty$ map  from $\mathcal{A}^{\Omega^{o}}_{\partial\Omega^{i}}\times C_0^{\frac{\alpha}{2};  \alpha}([0,T] \times \partial\Omega^{i})$ to $C_0^{\frac{1+\alpha}{2};  1+\alpha}([0,T] \times \partial\Omega^{o})$ that takes $(\phi,\mu) \in \mathcal{A}^{\Omega^{o}}_{\partial\Omega^{i}}\times C_0^{\frac{\alpha}{2};  \alpha}([0,T] \times \partial\Omega^{i})$ to
\begin{equation}\label{eq:V2}
V_2[\phi,\mu]:=  v^-_{\Omega^{i}[\phi]}[\mu  \circ (\phi^T)^{(-1)}]_{|[0,T]\times \partial \Omega^{o}} \, .
\end{equation}
\item We denote by $V_3$ the $C^\infty$ map from $\mathcal{A}^{\Omega^{o}}_{\partial\Omega^{i}} \times C_0^{\frac{\alpha}{2};  \alpha}([0,T] \times \partial\Omega^{o})$ to $C_0^{\frac{1+\alpha}{2};  1+\alpha}([0,T] \times \partial\Omega^{i})$ that takes $(\phi,\mu) \in \mathcal{A}^{\Omega^{o}}_{\partial\Omega^{i}}\times C_0^{\frac{\alpha}{2};  \alpha}([0,T] \times \partial\Omega^{o})$ to 
\begin{equation}\label{eq:V3}
V_3[\phi,\mu]:= v^+_{\Omega^{o}}[\mu ]_{|[0,T] \times \partial\Omega^{i}[\phi]}\circ \phi^T \, .
\end{equation}
\end{itemize}

Our next step is to compute the partial differentials of the maps $V_1$, $V_2$, and $V_3$ of \eqref{eq:V1}--\eqref{eq:V3}. The differentials of $V_2$ and $V_3$ can be obtained using standard calculus in Banach spaces, as they {only} involve non-singular integral operators. Instead, computing the differential of $V_1$ requires more care. {The proof is exactly the one} used in the proof of Lanza de Cristoforis and Rossi  \cite[Prop.~3.14]{LaRo04}, where, in place of their inequality (3.29), {here one needs to employ} the {inequality} 
\[
\left| \nabla_x S_n(t,x)\right| \leq K  t^{-\frac{n+1}{2} }e^{-\frac{|x|^2}{8t}}\quad\text{for all }(t,x)\in \mathopen]0,+\infty\mathclose [ \times \mathbb{R}^{n}\,,
\]
which holds  for some $K>0$ (see Lady\v{z}enskaja, Solonnikov and Ural'ceva \cite[p. 274]{LaSoUr68}). We also make use of \cite[Thm.~5.3]{DaLu23}, \cite[Lem.~7.1]{LaLu17}. The results are presented in the following Proposition \ref{prop:diff}, where statement (i) concerns  $V_1$, while statements (ii) and (iii) address $V_2$ and $V_3$, respectively.

\begin{proposition}\label{prop:diff}   Let $\alpha \in \mathopen]0,1[$ and $T>0$. Let $\Omega^{o}$, $\Omega^{i}$ be as in \eqref{introsetconditions}. Then the following {statements} hold.
\begin{itemize}
\item[(i)] If $(\phi_0,\mu_0) \in \mathcal{A}^{\Omega^{o}}_{\partial\Omega^{i}}\times C_0^{\frac{\alpha}{2};  \alpha}([0,T] \times \partial\Omega^{i})$ then the  partial  differential of $V_1[\cdot,\cdot ]$ at $(\phi_0,\mu_0)$ with respect to the variable $\phi$ is delivered by the formula
\begin{equation}\label{eq:diffV1}
\begin{split}
\partial_\phi & V_1[\phi_0,\mu_0][h](t,\xi)\\
&=\int_0^t\int_{\partial \Omega^{i}} \nabla  S_n(t-\tau,\phi_0(\xi)-\phi_0(s))\cdot \left(h(\xi)-h (s)\right) \tilde{\sigma}_n[\phi_0] (s)\mu_0(\tau,s)\,d\sigma_s\, d\tau\\
& \quad +\int_0^t\int_{\partial \Omega^{i}}  S_n(t-\tau,\phi_0(\xi)-\phi_0(s)) d \tilde{\sigma}_n[\phi_0][h] (s)\mu_0(\tau,s)\,d\sigma_s\, d\tau,
\end{split}
\end{equation}
for all $(t,\xi) \in [0,T] \times \partial\Omega^{i}$ and for all $h  \in C^{1,\alpha}(\partial \Omega^{i}, \mathbb{R}^n)$.
\item[(ii)] {If $(\phi_0,\mu_0)\in  \mathcal{A}^{\Omega^{o}}_{\partial\Omega^{i}}\times C_0^{\frac{\alpha}{2};  \alpha}([0,T] \times \partial\Omega^{i})$} then the partial  differential of $V_2[\cdot,\cdot ]$ at $(\phi_0,\mu_0)$ with respect to the variable $\phi$ is delivered by the formula
\begin{equation}\label{eq:diffV2}
\begin{split}
\partial_\phi & V_2[\phi_0,\mu_0][h](t,x)\\
&=-\int_0^t\int_{\partial \Omega^{i}} \nabla  S_n(t-\tau,x-\phi_0(s))\cdot h (s) \tilde{\sigma}_n[\phi_0] (s)\mu_0(\tau,s)\,d\sigma_s\, d\tau\\
& \quad +\int_0^t\int_{\partial \Omega^{i}}  S_n(t-\tau,x-\phi_0(s)) d \tilde{\sigma}_n[\phi_0][h] (s)\mu_0(\tau,s)\,d\sigma_s\, d\tau,
\end{split}\end{equation}
for all $(t,x) \in [0,T] \times \partial\Omega^{o}$ and for all $h \in C^{1,\alpha}(\partial \Omega^{i}, \mathbb{R}^n)$.
\item[(iii)] If $(\phi_0,\mu_0)\in \mathcal{A}^{\Omega^{o}}_{\partial\Omega^{i}} \times C_0^{\frac{\alpha}{2};  \alpha}([0,T] \times \partial\Omega^{o})$ then the partial  differential of $V_3[\cdot,\cdot ]$ at $(\phi_0,\mu_0)$ with respect to the variable $\phi$ is delivered by the formula
\begin{equation}\label{eq:diffV3}
\begin{split}
\partial_\phi & V_3[\phi_0,\mu_0][h](t,\xi)\\
&=\int_0^t\int_{\partial \Omega^{o}}\nabla  S_n(t-\tau,\phi_0(\xi)-y) \cdot h(\xi) \mu_0(\tau,y)\,d\sigma_y\, d\tau\, ,
\end{split}\end{equation}
for all $(t,\xi) \in [0,T] \times \partial\Omega^{i}$ and for all $h \in C^{1,\alpha}(\partial \Omega^{i}, \mathbb{R}^n)$.
\end{itemize}
\end{proposition}

\begin{remark}
The formulas for the differentials in Proposition~\ref{prop:diff} involve the differential of $\tilde{\sigma}_n[\cdot]$. While we do not make use of an explicit expression for this differential, we note here that it can, in fact, be computed. For the special case where $\Omega^i$ is a ball, we refer the reader to Lanza de Cristoforis and Rossi \cite[p.~173]{LaRo04}. For the general case, see Rossi \cite[p.~63]{Ro05}; see also Henrot and Pierre \cite[p.~197]{HePi18} and Costabel and Le Lou\"er \cite[Lemma 4.2]{CoLe12a}.

We also note that in Proposition \ref{prop:diff}, we restricted our attention to the first differentials of $V_1$, $V_2$, and $V_3$, but differentials of arbitrary order can, in fact, be computed. As before, for $V_2$ and $V_3$, this follows from standard calculus in Banach spaces, as these maps involve non-singular integral operators. For $V_1$, higher-order differentials can be obtained  following the argument used by Lanza de Cristoforis and Rossi in \cite[Prop.~3.14]{LaRo04}, and using the inequality 
\[
\left|D_x^\eta \partial_t^hS_n(t,x)\right| \leq K_{\eta,h} t^{-\frac{n}{2}-\frac{|\eta|_1}{2}-h}e^{-\frac{|x|^2}{8t}}\quad\text{for all }(t,x)\in \mathopen]0,+\infty\mathclose [ \times \mathbb{R}^{n}\,,
\]
which holds for some $K_{\eta,h} >0$ depending on $\eta \in \mathbb{N}^n$ and $h\in\mathbb{N}$ (see Lady\v{z}enskaja, Solonnikov and Ural'ceva \cite[p. 274]{LaSoUr68}). 
\end{remark}

 We recall that we are interested in characterizing the differential {of the normal derivative of the solution} with respect to the variable $\phi$. As a preliminary step, we need to characterize the differential $(d\mu^o[\phi],d\mu^i[\phi])$ of the map
\[
\phi \mapsto (\mu^o[\phi],\mu^i[\phi])\, .
\]
 {The strategy is to differentiate  { identity
 \eqref{M=0}, that is }
\[
\mathcal{M}[\phi](\mu^o[\phi],\mu^i[\phi])=(g^o,g^i)\, ,
\]
with respect to $\phi$, in order to obtain a characterization of the differentials
\[
(d\mu^o[\phi],d\mu^i[\phi])\, .
\]}

\begin{proposition}\label{prop:diffm}   
Let $\alpha \in \mathopen]0,1[$ and $T>0$. Let $\Omega^{o}$, $\Omega^{i}$ be as in \eqref{introsetconditions}.  Let $(g^o,g^i)$ be as in \eqref{introfunconditions}. Let $\phi_0\in \mathcal{A}^{\Omega^{o}}_{\partial\Omega^{i}}$. Let $h\in C^{1,\alpha}(\partial \Omega^{i},\mathbb{R}^n)$. Then $(d\mu^o[\phi_0][h],d\mu^i[\phi_0][h])$ is the unique pair in ${\mathcal{X}^{\frac{\alpha}{2};  \alpha}_0}$ such that
\begin{equation}\label{eq:diffm}   
\mathcal{M}[\phi_0](d\mu^o[\phi_0]{[h]},d\mu^i[\phi_0]{[h]})=(\tilde{g}^o,\tilde{g}^i)\, ,
\end{equation}
where
\[
\begin{split}
\tilde{g}^o(t,x):=&-\int_0^t\int_{\partial \Omega^{i}}S_n(t-\tau,x-\phi_0(s))d\tilde{\sigma}_n[\phi_0][h](s)\mu^i[\phi_0](\tau,s)\,d\sigma_s\, d\tau\\
&+\int_0^t\int_{\partial \Omega^{i}} \nabla S_n(t-\tau,x-\phi_0(s)) \cdot h(s)\tilde{\sigma}_n[\phi_0](s){\mu^i[\phi_0]}(\tau,s)\,d\sigma_s\, d\tau
\end{split}
\]
for all $(t,x)\in [0,T]\times \partial \Omega^{o}$ and
\[
\begin{split}
\tilde{g}^i(t,\xi):=&-\int_0^t\int_{\partial \Omega^{o}}\nabla S_n(t-\tau,\phi_0(\xi)-y) \cdot h(\xi)\mu^o[\phi_0](\tau,y)\,d\sigma_y\, d\tau\\
&-\int_0^t\int_{\partial \Omega^{i}}S_n(t-\tau,\phi_0(\xi)-\phi_0(s))d \tilde{\sigma}_n[\phi_0][h](s)\mu^i[\phi_0](\tau,s)\,d\sigma_s\, d\tau\\
&-\int_0^t\int_{\partial \Omega^{i}}\nabla  S_n(t-\tau,\phi_0(\xi)-\phi_0(s)) \cdot (h(\xi)-h(s)) \tilde{\sigma}_n[\phi_0](s)\mu^i[\phi_0](\tau,s)\,d\sigma_s\, d\tau\, 
\end{split}
\]
for all $(t,\xi)\in [0,T]\times \partial \Omega^{i}$.
\end{proposition}
\begin{proof}
By Proposition \ref{prop:diff} and by computing the differential with respect to $\phi$ at $\phi_0$ of equality
\[
\mathcal{M}[\phi](\mu^o[\phi],\mu^i[\phi])=(g^o,g^i)\, ,
\]
we obtain 
\[
\left\{
\begin{array}{ll}
 V_{\partial\Omega^{o}}  [d\mu^o[\phi_0][h]] + V_2[\phi_0,d\mu^i[\phi_0][h]]=-\partial_\phi V_2[\phi_0,\mu^i[\phi_0]][h] \, , &
\\
V_3[\phi_0,d\mu^o[\phi_0][h]]+V_1[\phi_0,d\mu^i[\phi_0][h]]=- \partial_\phi V_3[\phi_0,\mu^o[\phi_0]][h]- \partial_\phi V_1[\phi_0,\mu^i[\phi_0]][h]\, , &
\end{array}
\right.
\]
for all $h\in C^{1,\alpha}(\partial \Omega^{i},\mathbb{R}^n)$. Then, by formulas \eqref{eq:diffV1}--\eqref{eq:diffV3} and by the definition of $\tilde{g}^o$ and $\tilde{g}^i$, we deduce the validity of equation \eqref{eq:diffm}.
\end{proof}

\begin{remark}
Let $\alpha \in \mathopen]0,1[$ and $T>0$. Let $\Omega^{o}$, $\Omega^{i}$ be as in \eqref{introsetconditions}.  Let $(g^o,g^i)$ be as in \eqref{introfunconditions}. Let $\phi_0\in \mathcal{A}^{\Omega^{o}}_{\partial\Omega^{i}}$. Let $h\in C^{1,\alpha}(\partial \Omega^{i},\mathbb{R}^n)$. Let $(\tilde{g}^o,\tilde{g}^i)$ be as in Proposition \ref{prop:diffm}.
Then 
\[
(d\mu^o[\phi_0]{[h]},d\mu^i[\phi_0]{[h]})=\mathcal{M}[\phi_0]^{(-1)}(\tilde{g}^o,\tilde{g}^i)\, .
\]
\end{remark}

We are now in the position to compute the differential of the map
\[
\phi \mapsto \frac{\partial}{\partial \nu_{\Omega^{o}}}u^D_{\phi}
\]
with respect to the shape parameter $\phi$.

\begin{theorem}\label{thm:diffnormal}
Let $\alpha \in \mathopen]0,1[$ and $T>0$. Let $\Omega^{o}$, $\Omega^{i}$ be as in \eqref{introsetconditions}.  Let $(g^o,g^i)$ be as in \eqref{introfunconditions}. Let $\phi_0\in \mathcal{A}^{\Omega^{o}}_{\partial\Omega^{i}}$. Let $h\in C^{1,\alpha}(\partial \Omega^{i},\mathbb{R}^n)$. Let $(\tilde{g}^o,\tilde{g}^i)$ be as in Proposition \ref{prop:diffm}. Let
\[
(\mu^o_0,\mu^i_0)=\mathcal{M}[\phi_0]^{(-1)}(g^o,g^i)\, , \qquad (\mu^o_1,\mu^i_1)=\mathcal{M}[\phi_0]^{(-1)}(\tilde{g}^o,\tilde{g}^i)\, .
\]
Then
\[
d_\phi \Big(\frac{\partial}{\partial \nu_{\Omega^{o}}}u^D_{\phi}\Big) _{|{\phi=\phi_0}} [h]=  \frac{\partial}{\partial \nu_{\Omega^o}}\tilde{u}^D \qquad \text{on}\ [0,T]\times \partial \Omega^{o}\, ,
\]
where $\tilde{u}^D(t,x)$ is the unique continuous extension to $ [0,T]\times (\overline{\Omega^o}\setminus  \Omega^{i}[\phi_0])$ of
\[
\begin{split}
\tilde{u}^D(t,x):=&\,\,v^+_{\Omega^{o}}[\mu^o_1](t,x)+v^-_{\Omega^{i}[\phi_0]}[\mu^i_1 \circ (\phi_0^T)^{(-1)}](t,x)\\
&+\int_0^t\int_{\partial \Omega^{i}}S_n(t-\tau,x-\phi_0(s))d\tilde{\sigma}_n[\phi_0][h](s)\mu^i_0(\tau,s)\,d\sigma_s\, d\tau\\
&-\int_0^t\int_{\partial \Omega^{i}} \nabla S_n(t-\tau,x-\phi_0(s))\cdot h(s)\tilde{\sigma}_n[\phi_0](s)\mu^i_0(\tau,s)\,d\sigma_s\, d\tau\\
& \qquad \qquad\qquad \qquad \qquad \qquad \qquad \forall (t,x)\in [0,T]\times (\overline{\Omega^o}\setminus \overline {\Omega^{i}[\phi_0]})\, .
\end{split}
\]
\end{theorem}
\begin{proof}
By Proposition \ref{prop M=0},  for every $\phi  \in \mathcal{A}^{\Omega^{o}}_{\partial\Omega^{i}}$ we have 
 \[ 
    \begin{split}
        u^D_{\phi}(t,x) 
        &= \int_{0}^{t} \int_{\partial \Omega^{o}} S_{n}(t-\tau, x-y) \mu^o[\phi](\tau, y) \,d\sigma_y d\tau 
        \\
       & \qquad  +\int_{0}^{t} \int_{\partial \Omega^{i}} S_{n}(t-\tau, x-\phi(y)) \mu^i[\phi](\tau, y) \tilde{\sigma}_n[\phi](y) \,d\sigma_y d\tau \, ,
    \end{split}
\] 
    for all $(t,x) \in [0,T] \times (\overline{\Omega^{o}} \setminus \Omega^{i}[\phi])$. Then we note that
\[
    \begin{split}
       \frac{\partial}{\partial \nu_{\Omega^{o}}} u^D_{\phi}(t,x) 
        &=  \frac{1}{2}\mu^o[\phi](t,x) +W_{\partial\Omega^{o}}^* [\mu^o[\phi]](t,x) 
        \\
       & \qquad  +       \frac{\partial}{\partial \nu_{\Omega^{o}}}      u^{\#}[\phi](t,x) \, ,
    \end{split}
\]
    for all $(t,x) \in [0,T] \times \partial \Omega^{o}$, where
    \[
     u^{\#}[\phi](t,x):=\int_{0}^{t} \int_{\partial \Omega^{i}} S_{n}(t-\tau, x-\phi(y)) \mu^i[\phi](\tau, y) \tilde{\sigma}_n[\phi](y) \,d\sigma_y d\tau \, ,
    \]
    for all $(t,x) \in [0,T] \times (\overline{\Omega^{o}} \setminus \Omega^{i}[\phi])$. Then by standard rules of differentiation, we deduce the validity of the theorem.
\end{proof}

\begin{remark}
Under the assumptions of Theorem \ref{thm:diffnormal}, one verifies that ${\tilde{u}^D}$ is of class $C_{0}^{\frac{1+\alpha}{2}; 1+\alpha}$ close to $[0,T]\times \partial \Omega^o$ and thus it makes sense to consider the normal derivative {$\frac{\partial}{\partial \nu_{\Omega^o}}\tilde{u}^D$. In particular,} such normal derivative belongs to $C_{0}^{\frac{\alpha}{2}; \alpha}([0,T]\times \partial \Omega^o)$.
\end{remark}

Our approach is based  on potential theoretic methods,   thus the formula for the shape differential is given through integral equations and layer potential representations. Instead, in Chapko, Kress, and Yoon \cite{ChKrYo98}  the authors differentiate the weak formulations of the initial-boundary value problem, and thus obtain a description of the shape differential via a boundary value problem. {To compare our result with that of  Chapko, Kress, and Yoon \cite{ChKrYo98} we now characterize the function $\tilde{u}^D$ of Theorem \ref{thm:diffnormal} as the solution to a boundary value problem. This is done in the following Theorem \ref{cor:der}.}

\begin{theorem}\label{cor:der}
Let $\alpha \in \mathopen]0,1[$ and $T>0$. Let $\Omega^{o}$, $\Omega^{i}$ be as in \eqref{introsetconditions}.  {Let $(g^o,g^i)$ be as in \eqref{introfunconditions}}. Let $\phi_0\in \mathcal{A}^{\Omega^{o}}_{\partial\Omega^{i}}$. Let $h\in C^{1,\alpha}(\partial \Omega^{i},\mathbb{R}^n)$. 
Then
\[
d_\phi \Big(\frac{\partial}{\partial \nu_{\Omega^{o}}}u^D_{\phi}\Big)_{|{\phi=\phi_0}}  [h]=  \frac{\partial}{\partial \nu_{\Omega^o}}\tilde{u}^D \qquad \text{on}\ [0,T]\times \partial \Omega^{o}\, ,
\]
where {$\tilde u^D$} is the unique solution to problem 
\begin{equation}\label{shapesys}
\begin{cases}
    \partial_t u - \Delta u = 0 & \quad\text{ in } ]0,T] \times (\Omega^{o} \setminus \overline{\Omega^{i}[\phi_0]}), 
    \\
u = 0& \quad \mbox{ on }  [0,T] \times \partial \Omega^{o}, 
    \\
    u  = -\left(h \circ {\phi_0}^{(-1)}\right )\cdot  \nabla  u^D_{\phi_0}  & \quad {\mbox{ on } }  [0,T] \times \partial \Omega^{i}[\phi_0],
    \\
    u(0,\cdot)=0 & \quad \text{ in } \overline{\Omega^{o}} \setminus \Omega^{i}[\phi_0]\, .
    \end{cases}
\end{equation}
\end{theorem} 
\begin{proof}
{ First we observe that, since  $-\left(h \circ \phi_0^{(-1)}\right )\cdot  \nabla  u^D_{\phi_0}   \in C_0^{\frac{\alpha}{2};  \alpha}([0,T] \times \partial\Omega^{i}[\phi_0])$, problem \eqref{shapesys} admits a unique solution in $ C_{0}^{\frac{\alpha}{2}; \alpha}([0,T] \times (\overline{\Omega^{o}} \setminus \Omega^{i}[{\phi_0}]))$ (cf. e.g. Baderko \cite{Ba97}).  Also note that, since problem \eqref{shapesys} has a homogeneous---and hence smooth---boundary condition on the outer boundary, its unique solution is regular enough to take its normal derivative on this same outer boundary.  }
 Since $\tilde u^D${---as defined in Theorem \ref{thm:diffnormal}---}clearly solves the heat equation and {has} zero initial condition, {in order to check that it coincides with the unique solution to problem \eqref{shapesys}} we only need to compute which Dirichlet boundary conditions  {it satisfies}. {As a consequence, we need to characterize the trace of $\tilde{u}^D$ on $[0,T] \times \partial (\Omega^{o} \setminus \overline{\Omega^{i}[\phi_0]})$.} We start by computing the value of $\tilde u^D$ on $[0,T]\times\partial\Omega^o$: Let $(t,x) \in [0,T]\times\partial\Omega^o$, then by the definition of $\mathcal{M}$, the equation \eqref{eq:diffm} defining the shape differentials of $(\mu^o[\phi],\mu^i[\phi])$, and the definition of $\tilde u^D$ we get
\[
 \tilde u^D(t,x) = 0\,,
\]
as expected. Next we consider the value of $\tilde u^D$ on $[0,T]\times\partial\Omega^i[\phi_0]$. Let $(t,\xi) \in [0,T]\times\partial\Omega^i$,  then by arguing again as above and also using the jump formula for the gradient of the single layer potential (see Theorem \ref{thmsl} (iii))   we get  
\begin{align*}
 \tilde u^D(t,\phi_0(\xi))  =\,\,& v^+_{\Omega^{o}}[\mu^o_1](t,\phi_0(\xi))+V_{\partial\Omega^{i}[\phi_0]}[\mu^i_1 \circ (\phi_0^T)^{(-1)}](t,{\phi_0}(\xi))\\
 &+\int_0^t\int_{\partial \Omega^{i}}S_n(t-\tau,\phi_0(\xi)-\phi_0(s))d\tilde{\sigma}_n[\phi_0][h](s)\mu^i_0(\tau,s)\,d\sigma_s\, d\tau\\
&-\int_0^t\int_{\partial \Omega^{i}}  \nabla S_n(t-\tau,\phi_0(\xi)-\phi_0(s))\cdot h(s)\tilde{\sigma}_n[\phi_0](s)\mu^i_0(\tau,s)\,d\sigma_s\, d\tau\\
 &+\frac12 h(\xi) \cdot \nu_{\Omega^i[\phi_0]}(\phi_0(\xi))\mu_0^i(t,\xi)\\
= \,\,& \mathcal{M}_2[\phi_0](\mu^o_1,\mu^i_1)(t,\xi)\\
&+\int_0^t\int_{\partial \Omega^{i}}S_n(t-\tau,\phi_0(\xi)-\phi_0(s))d\tilde{\sigma}_n[\phi_0][h](s)\mu^i_0(\tau,s)\,d\sigma_s\, d\tau\\
&-\int_0^t\int_{\partial \Omega^{i}}  \nabla S_n(t-\tau,\phi_0(\xi)-\phi_0(s))\cdot h(s)\tilde{\sigma}_n[\phi_0](s)\mu^i_0(\tau,s)\,d\sigma_s\, d\tau\\
 &+\frac12 h(\xi) \cdot \nu_{\Omega^i[\phi_0]}(\phi_0(\xi))\mu_0^i(t,\xi)\\
= \,\,&-\int_0^t\int_{\partial \Omega^{o}}\nabla S_n(t-\tau,\phi_0(\xi)-y)\cdot h(\xi)\mu^o_0(\tau,y)\,d\sigma_y\, d\tau\\
&-\int_0^t\int_{\partial \Omega^{i}}S_n(t-\tau,\phi_0(\xi)-\phi_0(s))d \tilde{\sigma}_n[\phi_0][h](s)\mu^i_0(\tau,s)\,d\sigma_s\, d\tau\\
&-\int_0^t\int_{\partial \Omega^{i}}\nabla  S_n(t-\tau,\phi_0(\xi)-\phi_0(s))\cdot( h(\xi)-h(s)) \tilde{\sigma}_n[\phi_0](s)\mu^i_0(\tau,s)\,d\sigma_s\, d\tau\\
&+\int_0^t\int_{\partial \Omega^{i}}S_n(t-\tau,\phi_0(\xi)-\phi_0(s))d\tilde{\sigma}_n[\phi_0][h](s)\mu^i_0(\tau,s)\,d\sigma_s\, d\tau\\
&-\int_0^t\int_{\partial \Omega^{i}} \nabla S_n(t-\tau,\phi_0(\xi)-\phi_0(s))\cdot h(s) \tilde{\sigma}_n[\phi_0](s)\mu^i_0(\tau,s)\,d\sigma_s\, d\tau\\
 &+\frac12 h(\xi) \cdot \nu_{\Omega^i[\phi_0]}(\phi_0(\xi))\mu_0^i(t,\xi)\\
= \,\,&-\int_0^t\int_{\partial \Omega^{o}}\nabla S_n(t-\tau,\phi_0(\xi)-y)\cdot h(\xi)\mu^o_0(\tau,y)\,d\sigma_y\, d\tau\\
&-\int_0^t\int_{\partial \Omega^{i}}\nabla  S_n(t-\tau,\phi_0(\xi)-\phi_0(s))\cdot h(\xi) \tilde{\sigma}_n[\phi_0](s)\mu^i_0(\tau,s)\,d\sigma_s\, d\tau\\
 &+\frac12 h(\xi) \cdot \nu_{\Omega^i[\phi_0]}(\phi_0(\xi))\mu_0^i(t,\xi)\\
= \,\,&- h(\xi)\cdot\int_0^t\int_{\partial \Omega^{o}} \nabla S_n(t-\tau,\phi_0(\xi)-y) \mu^o_0(\tau,y)\,d\sigma_y\, d\tau\\
&- h(\xi)\cdot \int_0^t\int_{\partial \Omega^{i}}\nabla  S_n(t-\tau,\phi_0(\xi)-\phi_0(s)) \tilde{\sigma}_n[\phi_0](s)\mu^i_0(\tau,s)\,d\sigma_s\, d\tau\\
 &+\frac12 h(\xi) \cdot \nu_{\Omega^i[\phi_0]}(\phi_0(\xi))\mu_0^i(t,\xi).
\end{align*}
By the jump formula for the gradient of the single layer potential and by changing the variable in the second integral we get
\begin{align*}
 \tilde u^D(t,\phi_0(\xi))  =\,\,& -h(\xi)\cdot  \nabla v^+_{\Omega^o}[\mu^o[\phi_0]](t,\phi_0(\xi))\\
 & - \frac12 h(\xi) \cdot \nu_{\Omega^i[\phi_0]}(\phi_0(\xi))\mu_0^i(t,\xi)- h(\xi) \cdot \nabla  v^-_{\Omega^i[\phi_0]}[\mu^i[\phi_0]\circ (\phi_0^T)^{(-1)}](t,\phi_0(\xi))\\
  &+\frac12 h(\xi) \cdot \nu_{\Omega^i[\phi_0]}(\phi_0(\xi))\mu_0^i(t,\xi)\\
  =\,\,&-h(\xi) \cdot \nabla u^D_{\phi_0} (t,\phi_0(\xi)),
\end{align*}
that completes the proof.
\end{proof}

\begin{remark}
Under the assumptions of {Theorem} \ref{cor:der}, {by combining interior and boundary regularity for the heat equation,} one verifies that the solution of problem \eqref{shapesys} is regular enough to consider the normal derivative on $[0,T]\times \partial \Omega^o$. 
\end{remark}

\subsection{Comparison with Chapko, Kress{,} and Yoon  \cite{ChKrYo98}  }\label{ss:comparison}

In Chapko, Kress{,} and Yoon  \cite{ChKrYo98} the authors consider  problem \eqref{princeqpertu} with $g^i=0$, that is:
\begin{equation}\label{bvp:inv}
\begin{cases}
    \partial_t u - \Delta u = 0 & \quad\text{in } ]0,T] \times (\Omega^{o} \setminus \overline{\Omega^{i}[\phi]}), 
    \\
u(t,x) = g^o (t,x)& \quad \forall (t,x)\in [0,T] \times \partial \Omega^{o}, 
    \\
    u (t,x) = 0 & \quad \forall (t,x)\in  [0,T] \times \partial \Omega^{i}[\phi],
    \\
    u(0,\cdot)=0 & \quad \text{in } \overline{\Omega^{o}} \setminus \Omega^{i}[\phi]\, .
    \end{cases}
\end{equation}

Obviously, our Theorem \ref{thm:smoothrepbis} can be specialized to the case of problem \eqref{bvp:inv} and gives the smoothness of the map $(\phi,g^o) \mapsto \frac{\partial}{\partial \nu_{\Omega^o}}u^D_{\phi,g^o,0}$.  

Also, since the inner boundary condition on $ [0,T] \times \partial \Omega^{i}[\phi]$ is zero, the solution $u^D_{\phi,g^o,0}$ is constant on $ [0,T] \times \partial \Omega^{i}[\phi]$ and accordingly {the tangential gradient of $u^D_{\phi,g^o,0}(t,\cdot)$ on $\partial \Omega^{i}[\phi]$ vanishes for all $t \in [0,T]$, that is 
\[
\nabla_{\partial \Omega^{i}[\phi]} u^D_{\phi,g^o,0}:= \nabla u^D_{\phi,g^o,0} - \left(\nu_{\Omega^{i}[\phi]} \cdot \nabla u^D_{\phi,g^o,0}\right) \nu_{\Omega^{i}[\phi]}  = 0 \qquad \mbox{ on }  [0,T] \times \partial \Omega^{i}[\phi]\,.
\]}
 This consideration immediately implies that { Theorem}  \ref{cor:der} when applied to problem \eqref{bvp:inv} coincides with 
   \cite[Thm. 2.1]{ChKrYo98}. Indeed {the result of \cite{ChKrYo98}}, stated with our notation and functional setting (they assume $\Omega$ to be a { subset} of class $C^2$ in {$\mathbb{R}^2$} and set the problem   in Sobolev spaces), reads:
\begin{corollary} 
Let $\alpha \in \mathopen]0,1[$ and $T>0$. Let $\Omega^{o}$, $\Omega^{i}$ be as in \eqref{introsetconditions}.  Let $(g^o,g^i)$ be as in \eqref{introfunconditions} and $g^i=0$. Let $\phi_0\in \mathcal{A}^{\Omega^{o}}_{\partial\Omega^{i}}$. Let $h\in C^{1,\alpha}(\partial \Omega^{i},\mathbb{R}^n)$. 
Then
\[
d_\phi \Big(\frac{\partial}{\partial \nu_{\Omega^{o}}}u^D_{\phi}\Big)_{|{\phi=\phi_0}}  [h]=  \frac{\partial}{\partial \nu_{\Omega^o}}\tilde{u}^D \qquad \text{on}\ [0,T]\times \partial \Omega^{o}\, ,
\]
where $\tilde u^D$ is the unique solution to problem
\[ 
\begin{cases}
    \partial_t u - \Delta u = 0 & \quad\text{ in } ]0,T] \times (\Omega^{o} \setminus \overline{\Omega^{i}[\phi_0]}), 
    \\
u = 0& \quad \mbox{ on }  [0,T] \times \partial \Omega^{o}, 
    \\
    u  = -\left(h \circ {\phi_0^{(-1)}}\right )\cdot \nu_{\Omega^i[\phi_0]}\frac{\partial}{\partial \nu_{\Omega^i[\phi_0]}}u^D_{\phi_0}   & \quad {\mbox{ on } }  [0,T] \times \partial \Omega^{i}[\phi_0],
    \\
    u(0,\cdot)=0 & \quad \text{ in } \overline{\Omega^{o}} \setminus \Omega^{i}[\phi_0]\, .
    \end{cases}
\] 
\end{corollary} 

 \section{The {perturbed Neumann problem  \eqref{princeqpertuneu}}} \label{s:neum}

{
Let  $\Omega^{o}$ and $\Omega^{i}$  satisfy condition \eqref{introsetconditions}, $\phi \in  \mathcal{A}^{\Omega^{o}}_{\partial\Omega^{i}}$, and 
\[
(f^o,f^i) \in \mathcal{X}_0^{\frac{\alpha}{2}; \alpha} \, .
\]
In this section we   consider the  Neumann  problem \eqref{princeqpertuneu}.} As  is well known, for each triple
\[
(\phi,f^o, f^i ) \in \mathcal{A}^{\Omega^{o}}_{\partial\Omega^{i}}\times   \mathcal{X}_0^{\frac{\alpha}{2}; \alpha}
\]
problem \eqref{princeqpertuneu} has a unique solution in $C_{0}^{\frac{1+\alpha}{2}; 1+\alpha}([0,T] \times (\overline{\Omega^{o}} \setminus \Omega^{i}[\phi]))$, and we denote such a solution by $u^N_{\phi,f^o, f^i}$ {(see e.g. Baderko \cite[Thm. 5.1]{Ba97})}. 

As before, our approach is based on layer potential representations and consequently integral equations. To this aim, 
 we define the map $\mathcal{N}=(\mathcal{N}_1,\mathcal{N}_2)$  from $\mathcal{A}^{\Omega^{o}}_{\partial\Omega^{i}}$ to $\mathcal{L}(\mathcal{X}_0^{\frac{\alpha}{2}; \alpha},\mathcal{X}_0^{\frac{\alpha}{2}; \alpha})$ which takes $\phi \in \mathcal{A}^{\Omega^{o}}_{\partial\Omega^{i}}$ to the linear and continuous operator from $\mathcal{X}_0^{\frac{\alpha}{2}; \alpha}$ to $\mathcal{X}_0^{\frac{\alpha}{2}; \alpha}$ given by
\begin{equation}\label{N}
\begin{aligned}
\mathcal{N}_1[\phi](\mu^o,\mu^i) & := \frac{1}{2}\mu^o+W^*_{\partial\Omega^o}[\mu^o]+\partial_{\nu_{\Omega^o}} v^-_{ \Omega^i[\phi]}[\mu^i \circ (\phi^T)^{(-1)}]  \qquad &{\mbox{ on } \, [0,T] \times \partial\Omega^{o}},
\\
\mathcal{N}_2[\phi](\mu^o,\mu^i) & := -\frac12 \mu^i +W^*_{\partial\Omega^i[\phi]}[\mu^i \circ (\phi^T)^{(-1)}]\circ \phi^T +\partial_{\nu_{\Omega^i[\phi]}} v^+_{ \Omega^o}[\mu^o]  \circ \phi^T
 &{\mbox{ on } \, [0,T] \times \partial\Omega^{i}},
\end{aligned}
\end{equation}
for all $(\mu^o,\mu^i) \in \mathcal{X}_0^{\frac{\alpha}{2};  \alpha}$.  

From the definition of $\mathcal{N}$ and the properties of the single layer potential of Theorem \ref{thmsl}, we  deduce the validity of the following result regarding the solvability of problem  \eqref{princeqpertuneu} by means of single layer potentials.  For more details we refer \cite[\S 4]{DaLuMoMu25}, that concerns  a problem with a  Robin boundary condition on the hole, having the Neumann problem \eqref{princeqpertuneu} with $f^i=0$ as a particular case. The general case is a straightforward modification.
\begin{proposition}\label{prop N=0}
Let $\alpha \in \mathopen]0,1[$ and $T>0$. Let $\Omega^{o}$, $\Omega^{i}$ be as in \eqref{introsetconditions}. Let
\begin{equation*}
(\phi,\mu^o,\mu^i, f^o, f^i ) \in \mathcal{A}^{\Omega^{o}}_{\partial\Omega^{i}}\times \mathcal{X}_0^{\frac{\alpha}{2}; \alpha} \times {\mathcal{X}^{\frac{\alpha}{2}; \alpha}_0}\, .
\end{equation*}
{Then the function 
\[
(v^+_{\Omega^{o}} [\mu^o] + v^-_{\Omega^{i}[\phi]}[\mu^i\circ (\phi^T)^{(-1)}])_{| [0,T] \times (\overline{\Omega^{o}} \setminus {\Omega^{i}[\phi]})}
\]}  
is a solution of problem \eqref{princeqpertuneu} if and only if 
\begin{equation}\label{N=0}
\mathcal{N}[\phi](\mu^o,\mu^i)=(f^o, f^i).
\end{equation}
Moreover, for each $(\phi, f^o, f^i ) \in \mathcal{A}^{\Omega^{o}}_{\partial\Omega^{i}}\times {\mathcal{X}^{\frac{\alpha}{2}; \alpha}_0}$, there exists a unique pair $(\mu^o,\mu^i)\in \mathcal{X}_0^{\frac{\alpha}{2}; \alpha}$ such that \eqref{N=0} holds true and we denote such a pair by $(\mu^o[\phi, f^o, f^i],\mu^i[\phi, f^o, f^i])$.
\end{proposition}

 {Then, by \cite[\S 4]{DaLuMoMu25}, we deduce the validity of the following result.
 \begin{proposition}\label{prop1neu}
Let $\alpha \in \mathopen]0,1[$ and $T>0$. Let $\Omega^{o}$, $\Omega^{i}$ be as in \eqref{introsetconditions}. Then:
\begin{itemize}
\item[(i)] The map from $\mathcal{A}^{\Omega^{o}}_{\partial\Omega^{i}}\times {\mathcal{X}^{\frac{\alpha}{2}; \alpha}_0}$ to $\mathcal{X}_0^{\frac{\alpha}{2}; \alpha}$ which takes $(\phi, f^o, f^i)$ to $(\mu^o[\phi, f^o, f^i],\mu^i[\phi, f^o, f^i])$ is of class $C^\infty$.
\item[(ii)] The map from $\mathcal{A}^{\Omega^{o}}_{\partial\Omega^{i}}\times {\mathcal{X}^{\frac{\alpha}{2}; \alpha}_0}$ to $C_{0}^{\frac{1+\alpha}{2}; 1+\alpha}([0,T] \times \partial {\Omega^{o}})$ which takes $(\phi, f^o, f^i)$ to ${u^N_{\phi,f^o,f^i}}_{|[0,T] \times \partial \Omega^o}$ is of class $C^\infty$.
\end{itemize}
\end{proposition}
 }

 Motivated by the application to inverse heat conduction problems (see Chapko, Kress and Yoon \cite{ChKrYo99}), the next step is to compute the shape differential of the map $\phi \mapsto {u^N_{\phi,f^o,f^i}}_{|[0,T] \times \partial \Omega^o}$.   
From now on,  we consider problem \eqref{princeqpertuneu} with $f^o$ and $f^i$ fixed and we focus on the dependence on the {shape parameter} $\phi$. As {before}, we simply denote by $u^N_{\phi}$ the solution of problem \eqref{princeqpertuneu} and by    $(\mu^o[\phi],\mu^i[\phi])$  the solution to equation \eqref{N=0}.

For the sake of simplicity, we introduce some functionals related to layer potential operators involved in the definition of $\mathcal{N}$.
  Let $\alpha \in \mathopen]0,1[$ and $T>0$. Let $\Omega^{o}$, $\Omega^{i}$ be as in \eqref{introsetconditions}.

\begin{itemize}
\item We denote by $N_1$ the $C^\infty$ map from $\mathcal{A}^{\Omega^{o}}_{\partial\Omega^{i}}\times C_0^{\frac{\alpha}{2}; \alpha}([0,T] \times \partial\Omega^{i})$ to $C_0^{\frac{\alpha}{2}; \alpha}([0,T] \times \partial\Omega^{i})$ { taking} $(\phi,\mu) \in \mathcal{A}^{\Omega^{o}}_{\partial\Omega^{i}}\times C_0^{\frac{\alpha}{2}; \alpha}([0,T] \times \partial\Omega^{i})$ to 
\begin{equation}\label{eq:V1n}
N_1[\phi,\mu]:= W^*_{\partial\Omega^i[\phi]}[\mu \circ (\phi^T)^{(-1)}]\circ \phi^T \, .
\end{equation}
\item We denote by $N_2$ the $C^\infty$ map  from $\mathcal{A}^{\Omega^{o}}_{\partial\Omega^{i}}\times C_0^{\frac{\alpha}{2}; \alpha}([0,T] \times \partial\Omega^{i})$ to $C_0^{\frac{\alpha}{2}; \alpha}([0,T] \times \partial\Omega^{o})$ { taking} $(\phi,\mu) \in \mathcal{A}^{\Omega^{o}}_{\partial\Omega^{i}}\times C_0^{\frac{\alpha}{2}; \alpha}([0,T] \times \partial\Omega^{i})$ to
\begin{equation}\label{eq:V2n}
N_2[\phi,\mu]:=  \partial_{\nu_{\Omega^o}}  v^-_{ \Omega^i[\phi]}[\mu \circ (\phi^T)^{(-1)}]  \, .
\end{equation}
\item We denote by $N_3$ the $C^\infty$ map from $\mathcal{A}^{\Omega^{o}}_{\partial\Omega^{i}} \times C_0^{\frac{\alpha}{2};  \alpha}([0,T] \times \partial\Omega^{o})$ to $C_0^{\frac{\alpha}{2}; \alpha}([0,T] \times \partial\Omega^{i})$ { taking} $(\phi,\mu) \in \mathcal{A}^{\Omega^{o}}_{\partial\Omega^{i}}\times C_0^{\frac{\alpha}{2};  \alpha}([0,T] \times \partial\Omega^{o})$ to 
\begin{equation}\label{eq:V3n}
N_3[\phi,\mu]:=\left(\partial_{\nu_{\Omega^i[\phi]}}  v^+_{ \Omega^o}[\mu]_{|[0,T] \times \partial\Omega^{i}[\phi]}\right)  \circ \phi^T \, .
\end{equation}
\end{itemize}

Note that the smoothness of the above defined operators can be deduced from \cite{DaLu23}.
 We recall that we are interested in characterizing the differential  with respect to the variable $\phi$ of the shape functional associated with the solution $u^N_\phi$ on $[0,T] \times \partial \Omega^o$. As a preliminary step, we need to characterize the differential $(d\mu^o[\phi],d\mu^i[\phi])$ of 
\[
\phi \mapsto (\mu^o[\phi],\mu^i[\phi])
\]
with respect to the variable $\phi$. As an immediate consequence of Propositions \ref{prop N=0} and \ref{prop1neu} (i) we have the following.

\begin{proposition}\label{prop:diffmneu}   
Let $\alpha \in \mathopen]0,1[$ and $T>0$. Let $\Omega^{o}$, $\Omega^{i}$ be as in \eqref{introsetconditions}.  Let $(f^o,f^i)$ be as in \eqref{introfunconditionsneu}. Let $\phi_0\in \mathcal{A}^{\Omega^{o}}_{\partial\Omega^{i}}$. Let $h\in C^{1,\alpha}(\partial \Omega^{i},\mathbb{R}^n)$. Then $(d\mu^o[\phi_0][h],d\mu^i[\phi_0][h])$ is the unique pair in ${\mathcal{X}^{\frac{\alpha}{2}; \alpha}_0}$ such that
\begin{equation}\label{eq:diffmneu}   
\mathcal{N}[\phi_0](d\mu^o[\phi_0]{[h]},d\mu^i[\phi_0]{[h]})=(\tilde{f}^o,\tilde{f}^i)\, ,
\end{equation}
where
\[
\tilde{f}^o:=-{ \partial_{\phi} \Big(N_2[\phi,\mu^i[\phi_0]]\Big)_{|\phi=\phi_0}[h]} \qquad \mbox{ on }  [0,T]\times \partial \Omega^{o}
\]
 and
\[
\tilde{f}^i:=- {\partial_{\phi} \Big(N_1[\phi,\mu^i[\phi_0]]\Big)_{|\phi=\phi_0}[h]} -{\partial_{\phi} \Big(N_3[\phi,\mu^o[\phi_0]]\Big)_{|\phi=\phi_0}[h]}  \qquad \mbox{ on }  [0,T]\times \partial \Omega^{i}.
\]
\end{proposition}
 
\begin{remark}\label{rem:diffneum}
Let $\alpha \in \mathopen]0,1[$ and $T>0$. Let $\Omega^{o}$, $\Omega^{i}$ be as in \eqref{introsetconditions}.  Let $(f^o,f^i)$ be as in \eqref{introfunconditionsneu}. Let $\phi_0\in \mathcal{A}^{\Omega^{o}}_{\partial\Omega^{i}}$. Let $h\in C^{1,\alpha}(\partial \Omega^{i},\mathbb{R}^n)$. Let $(\tilde{f}^o,\tilde{f}^i)$ be as in Proposition \ref{prop:diffmneu}.
Then 
\[
(d\mu^o[\phi_0]{[h]},d\mu^i[\phi_0]{[h]})=\mathcal{N}[\phi_0]^{(-1)}(\tilde{f}^o,\tilde{f}^i)\, .
\]
\end{remark}

We are now able to characterize  the differential of the map $\phi \mapsto {u^N_\phi }_{|[0,T] \times \partial \Omega^o}$ via integral representation, in the same spirit we have already done for the Dirichlet case in Theorem \ref{thm:diffnormal}.

\begin{theorem}\label{thm:diffneum}
Let $\alpha \in \mathopen]0,1[$ and $T>0$. Let $\Omega^{o}$, $\Omega^{i}$ be as in \eqref{introsetconditions}.  Let $(f^o,f^i)$ be as in \eqref{introfunconditionsneu}. Let $\phi_0\in \mathcal{A}^{\Omega^{o}}_{\partial\Omega^{i}}$. Let $h\in C^{1,\alpha}(\partial \Omega^{i},\mathbb{R}^n)$. Let $(\tilde{f}^o,\tilde{f}^i)$ be as in Proposition \ref{prop:diffmneu}. Let
\[
(\mu^o_0,\mu^i_0)=\mathcal{N}[\phi_0]^{(-1)}(f^o,f^i)\, , \qquad (\mu^o_1,\mu^i_1)=\mathcal{N}[\phi_0]^{(-1)}(\tilde{f}^o,\tilde{f}^i)\, .
\]
Then
\[
d_\phi \Big( {u^N_{\phi}}_{|[0,T] \times \partial \Omega^o}\Big)_{{|\phi=\phi_0}} [h]=   \tilde{v}^N \qquad \text{on}\ [0,T]\times \partial \Omega^{o}\, ,
\]
where $\tilde{v}^N$ is defined by 
\begin{align*}
\tilde{v}^N:=\,\,v^+_{\Omega^{o}}[\mu^o_1]+v^-_{\Omega^{i}[\phi_0]}[\mu^i_1 \circ (\phi_0^T)^{(-1)}]
+d_{\phi} &\left(v^-_{\Omega^i[\phi]}[\mu_0^i\circ (\phi^T)^{(-1)}]\right)_{{|\phi=\phi_0}}[h]
\\& \qquad \qquad\mbox{ in } [0,T]\times (\overline{\Omega^o}\setminus \overline {\Omega^{i}[\phi_0]})\, .
\end{align*}
\end{theorem}
\begin{proof}
By Proposition \ref{prop N=0},  for every $\phi  \in \mathcal{A}^{\Omega^{o}}_{\partial\Omega^{i}}$ we have 
 \[ 
    \begin{split}
        u^N_{\phi}(t,x) 
        &= \int_{0}^{t} \int_{\partial \Omega^{o}} S_{n}(t-\tau, x-y) \mu^o[\phi](\tau, y) \,d\sigma_y d\tau 
        \\
       & \qquad  +\int_{0}^{t} \int_{\partial \Omega^{i}} S_{n}(t-\tau, x-\phi(s)) \mu^i[\phi](\tau, s) \tilde{\sigma}_n[\phi](s) \,d\sigma_s d\tau \, ,
    \end{split}
\] 
    for all $(t,x) \in [0,T] \times (\overline{\Omega^{o}} \setminus \Omega^{i}[\phi])$, where $(\mu^o[\phi],\mu^i[\phi])=\mathcal{N}[\phi]^{(-1)}(f^o,f^i)$.  Then by standard rules of differentiation, we deduce the validity of the theorem. 
    \end{proof}

As we have done for the Dirichlet problem in Section \ref{s:inv},  we also want to  obtain a result in the spirit of Chapko, Kress and Yoon \cite{ChKrYo99} . That is, we want to characterize the differential of $\phi \mapsto{u^N_{\phi}}_{|[0,T] \times \partial \Omega^o}$ via a boundary value problem.  To do so, we need to understand which Neumann problem satisfies the function $\tilde{v}^N$ defined in the above Theorem \ref{thm:diffneum}.  Note that to characterize the differential in this way, we will need more regularity: the hole and the linearized variable $h$ must be of class $C^{2,\alpha}$, and the boundary data $f^i$ must be of class $C_0^{\frac{1+\alpha}{2};1+ \alpha}([0,T]\times \partial\Omega^i)$. { The main reason for these {assumptions is} to have a solution of the Neumann problem that is $C^{\frac{2+\alpha}{2};2+\alpha}$ up to the boundary of the hole.}

 In the sequel, for the sake of simplicity we will denote by
\[
\nu_\phi := \nu_{\Omega^i[\phi]}\circ \phi.
\]
As a consequence of Theorem \ref{thm:diffneum} we are able to obtain the following, where, for a vector field $H$ defined on the boundary of an open set $\partial\Omega$,  we denote by $\nabla_{\partial\Omega} H$  the matrix whose $i$-th column is the tangential gradient of the $i$-th component of $H$.

\begin{theorem}\label{cor:derneu}
Let $\alpha \in \mathopen]0,1[$ and $T>0$. Let $\Omega^{o}$, $\Omega^{i}$ be as in \eqref{introsetconditions}.  Let $  (f^o,f^i) \in \mathcal{X}_0^{\frac{\alpha}{2};\alpha}$ and $f^i \in C_0^{\frac{1+\alpha}{2};1+ \alpha}([0,T]\times \partial\Omega^i)$. Let $\phi_0\in \mathcal{A}^{\Omega^{o}}_{\partial\Omega^{i}}$. Assume that $\phi_0(\partial\Omega^i)$ is of class $C^{2,\alpha}$ and   $h \circ \phi_0^{-1}\in C^{2,\alpha}(\phi_0(\partial\Omega^i),\mathbb{R}^n)$. 
Then
\[
d_\phi \Big( {u^N_{\phi}}_{|[0,T] \times \partial \Omega^o}\Big)_{{|\phi=\phi_0}} [h]=  \tilde{v}^N \qquad \text{on}\ [0,T]\times \partial \Omega^{o}\, ,
\]
where $\tilde{v}^N$ is the unique solution to problem 
\begin{equation}\label{shapesysneum}
\begin{cases}
    \partial_t u - \Delta u = 0 & \text{ in } ]0,T] \times (\Omega^{o} \setminus \overline{\Omega^{i}[\phi_0]}), 
    \\
\frac{\partial}{\partial \nu_{\Omega^o}}u = 0&  \mbox{ on }  [0,T] \times \partial \Omega^{o}, 
    \\
   \frac{\partial}{\partial \nu_{\Omega^i[\phi_0]}} u  =  - \left(h\circ \phi_0^{(-1)}\right)^\top \left(\nabla^2u^N_{\phi_0} \right)\nu_{\Omega^i[\phi_0]}\\\hspace{2cm}+\nabla u^N_{\phi_0} \cdot  \left((\nabla_{\partial\Omega^i[\phi_0]}(h \circ \phi_0^{(-1)})) \nu_{\Omega^i[\phi_0]}\right)   &  \mbox{ on }   [0,T] \times \partial \Omega^{i}[\phi_0],
    \\
    u(0,\cdot)=0 &  \text{ in } \overline{\Omega^{o}} \setminus \Omega^{i}[\phi_0]\, .
    \end{cases}
\end{equation}
\end{theorem} 
\begin{proof}
 First we observe that, under our regularity assumptions, it is classical that $u^N_{\phi_0}$ is of class $C^{\frac{2+\alpha}{2}; 2+\alpha}$ in a neighborhood of $[0,T] \times \partial\Omega^{i}[\phi_0]$ {(see Lady\v{z}enskaja,  Solonnikov, and  Ural'ceva \cite[Thm. 5.3 p. 320]{LaSoUr68}).}
  Then 
  \[
   - \left(h\circ \phi_0^{(-1)}\right)^\top \nabla^2u^N_{\phi_0} \nu_{\Omega^i[\phi_0]}+\nabla u^N_{\phi_0} \cdot  \left((\nabla_{\partial\Omega^i[\phi_0]}(h \circ \phi_0^{(-1)})) \nu_{\Omega^i[\phi_0]}\right)  \in C_0^{\frac{\alpha}{2};  \alpha}([0,T] \times \partial\Omega^{i}[\phi_0]).
  \] 
  Then, it is again classical that  problem \eqref{shapesysneum} admits a unique solution in $ C_{0}^{\frac{1+\alpha}{2}; 1+\alpha}([0,T] \times (\overline{\Omega^{o}} \setminus \Omega^{i}[{\phi_0}]))$ {(see e.g. Baderko \cite[Thm. 5.1]{Ba97})}. 
 Note that $\tilde{v}^N$---as defined in Theorem \ref{thm:diffneum}---clearly solves the heat equation and has zero initial condition.
 Also, one can check that under our assumptions $\tilde{v}^N \in  C_{0}^{\frac{1+\alpha}{2}; 1+\alpha}([0,T] \times (\overline{\Omega^{o}} \setminus \Omega^{i}[{\phi_0}]))$. Indeed, if   $(\mu^o_0,\mu^i_0)$ and $(\mu^o_1,\mu^i_1)$ are as in Theorem \ref{thm:diffneum}, under our assumption on the regularity of the domain and the boundary data, we have that 
 \[
\mu^i_0 \in C_0^{\frac{1+\alpha}{2}; 1+\alpha}([0,T]\times \partial\Omega^i). 
 \]
 Indeed, by using the invertibility of the operator 
 \[
 -\frac{1}{2}I+W^*_{\partial\Omega^i[\phi_0]}[\cdot] : C_0^{\frac{1+\alpha}{2}; 1+\alpha}([0,T]\times \partial\Omega^i[\phi_0]) \to C_0^{\frac{1+\alpha}{2}; 1+\alpha}([0,T]\times \partial\Omega^i[\phi_0])\,,
 \]
 which is an immediate consequence of the compactness of $W^*_{\partial\Omega^i[\phi_0]}$ in $C_0^{\frac{1+\alpha}{2}; 1+\alpha}([0,T]\times \partial\Omega^i[\phi_0])$ of \cite[Cor. 2.18(ii)]{Lu19} (see also \cite[Thm. 4.8(ii)]{LaLu19}),  one can bootstrap the regularity of $\mu^i_0$ noting that the term $\partial_{\nu_{\Omega^i[\phi]}} v^+_{ \Omega^o}[\mu^o]$ in the definition of $\mathcal{N}_2$---see \eqref{N}---is of class  $ C_0^{\frac{1+\alpha}{2}; 1+\alpha}([0,T]\times \partial\Omega^i[\phi_0])$.
 Moreover, recalling that
\begin{align*}
\tilde{v}^N:=&\,\,v^+_{\Omega^{o}}[\mu^o_1]+v^-_{\Omega^{i}[\phi_0]}[\mu^i_1 \circ (\phi_0^T)^{(-1)}]
+d_{\phi} \left(v^-_{\Omega^i[\phi]}[\mu_0^i\circ (\phi^T)^{(-1)}]\right)_{{|\phi=\phi_0}}[h]
\quad \mbox{ in } [0,T]\times (\overline{\Omega^o}\setminus \overline {\Omega^{i}[\phi_0]})\, ,
\end{align*}
we note that the first two terms in the right hand side of the previous equation are of class $C_{0}^{\frac{1+\alpha}{2}; 1+\alpha}([0,T] \times (\overline{\Omega^{o}} \setminus \Omega^{i}[\phi]))$ by Theorem \ref{thmsl}. Also, for $(t,x) \in [0,T]\times (\overline{\Omega^o}\setminus \overline {\Omega^{i}[\phi_0]})$ the map:
\begin{align*}
    d_{\phi}& \left(v^-_{\Omega^i[\phi]}[\mu_0^i\circ (\phi^T)^{(-1)}]\right)_{{|\phi=\phi_0}}[h](t,x) \\
    = &\int_0^t\int_{\partial \Omega^{i}}S_n(t-\tau,x-\phi_0(s))d\tilde{\sigma}_n[\phi_0][h](s)\mu^i_0(\tau,s)\,d\sigma_s\, d\tau\\
&-\int_0^t\int_{\partial \Omega^{i}} \nabla S_n(t-\tau,x-\phi_0(s))\cdot h(s)\tilde{\sigma}_n[\phi_0](s)\mu^i_0(\tau,s)\,d\sigma_s\, d\tau\\
  = &\int_0^t\int_{\partial \Omega^{i}}S_n(t-\tau,x-\phi_0(s))d\tilde{\sigma}_n[\phi_0][h](s)\mu^i_0(\tau,s)\,d\sigma_s\, d\tau\\
&-\sum_{i=1}^n\frac{\partial}{\partial x_i}\int_0^t\int_{\partial \Omega^{i}}  S_n(t-\tau,x-\phi_0(s)) h_i(s)\tilde{\sigma}_n[\phi_0](s)\mu^i_0(\tau,s)\,d\sigma_s\, 
    \end{align*}
     again by Theorem \ref{thmsl} admits {a} unique extension to a map in $C_{0}^{\frac{1+\alpha}{2}; 1+\alpha}([0,T] \times (\overline{\Omega^{o}} \setminus \Omega^{i}[{\phi_0}]))$.  Note that the previous equalities are justified since one can first differentiate in $\phi$ under the integral sign and then take out the space derivative since the kernel is non-singular and smooth for $(t,x) \in [0,T]\times (\overline{\Omega^o}\setminus \overline {\Omega^{i}[\phi_0]})$.

{Now,    to check that $\tilde{v}^N$ coincides with the unique solution to problem \eqref{shapesysneum} we only need to determine what Neumann boundary conditions  it satisfies.}  
 
We begin with the boundary condition on $[0,T] \times \partial \Omega^o$. By {the definition of $\tilde{v}^N$, by the jump formula for the normal derivative of the single layer potential of Theorem \ref{thmsl} (iii), and by}   Proposition \ref{prop:diffmneu} we have:
\begin{align*}
\partial_{\nu_{\Omega^o}} \tilde{v}^N =\,\,& \frac12 \mu^o_1 + W^*_{\partial \Omega^o}[\mu^o_1]+ \partial_{\nu_{\Omega^o}}v^-_{\Omega^{i}[\phi_0]}[\mu^i_1 \circ (\phi_0^T)^{(-1)}]\\
&+\partial_{\nu_{\Omega^o}} \left(d_{\phi} \left(v^-_{\Omega^i[\phi]}[\mu_0^i\circ (\phi^T)^{(-1)}]\right)_{{|\phi=\phi_0}}[h]\right)\\
=\,\,& \mathcal{N}_1[\phi_0](\mu^o_1,\mu^i_1) \\
&+\partial_\phi N_2[\phi_0,\mu_0^i][h] \\
=\,\,& 0.
\end{align*}
Note that we have used {the fact} that 
\[
\partial_{\nu_{\Omega^o}} \left(d_{\phi} \left(v^-_{\Omega^i[\phi]}[\mu_0^i\circ (\phi^T)^{(-1)}]\right)_{{|\phi=\phi_0}}[h]\right) =  d_{\phi} \left(\partial_{\nu_{\Omega^o}}  v^-_{ \Omega^i[\phi]}[\mu_0^i \circ (\phi^T)^{(-1)}] \right)_{{|\phi=\phi_0}}[h] =\partial_\phi N_2[\phi_0,\mu_0^i][h] 
\]
which can be easily justified since {the restriction of the single layer potential to} $[0,T] \times \partial \Omega^o$ has a non-singular smooth kernel.

Next we consider the boundary condition on $[0,T] \times  {\phi_0(\partial\Omega^i)}$. We will proceed  following the method of Potthast  \cite{Po94, Po96a} {(in this paper he considers the shape differential of solutions to the acoustic scattering problem)}. Let $(t,\xi) \in [0,T] \times \partial \Omega^i$. Then

\begin{align}\label{eqneumtot}
\nu_{\Omega^i[\phi_0]}(\phi_0(\xi))& \cdot \nabla  \tilde{v}^N(t,\phi_0(\xi)) \\\nonumber
= &\,\mathcal{N}_2[\phi_0](\mu^o_1,\mu^i_1)(t,\xi) \\ \nonumber
&+ \nu_{\Omega^i[\phi_0]}(\phi_0(\xi)) \cdot \nabla \left(d_{\phi} \left(v^-_{\Omega^i[\phi]}[\mu_0^i\circ (\phi^T)^{(-1)}]\right)_{{|\phi=\phi_0}}[h]\right)(t,\phi_0(\xi))\\ \nonumber
= &\,-\partial_\phi N_1[\phi_0,\mu_0^i][h] -\partial_\phi N_3[\phi_0,\mu_0^o][h]  \\ \nonumber
&+\nu_{\Omega^i[\phi_0]}(\phi_0(\xi)) \cdot \nabla \left(d_{\phi} \left(v^-_{\Omega^i[\phi]}[\mu_0^i\circ (\phi^T)^{(-1)}]\right)_{{|\phi=\phi_0}}[h]\right)(t,\phi_0(\xi)),
 \end{align}
 {where, with an abuse of notation, we have still denoted by $d_{\phi} \left(v^-_{\Omega^i[\phi]}[\mu_0^i\circ (\phi^T)^{(-1)}]\right)_{{|\phi=\phi_0}}[h]$ its extension to $[0,T]\times (\overline{\Omega^o}\setminus  {\Omega^{i}[\phi_0]})$. As we have already noted, such an extension is of class $C_{0}^{\frac{1+\alpha}{2}; 1+\alpha}([0,T] \times (\overline{\Omega^{o}} \setminus \Omega^{i}[\phi_0]))$. Let us compute the last term in the above equality.
To do this, we use the following notation. For $\xi \in \partial\Omega^i$ and $r>0$  small enough,} we set
\[
\xi^r_\phi=\phi(\xi)+r\nu_\phi(\xi)\in  (\overline{\Omega^o}\setminus \overline {\Omega^{i}[\phi]}).
\]
 {The} normal derivative we want to compute is:
\begin{align*}
\frac{d}{dr}\left(d_{\phi}\left( v^-_{\Omega^i[\phi]}[\mu_0^i\circ (\phi^T)^{(-1)}]\right)_{{|\phi=\phi_0}}[h](t,\xi^r_{\phi_0})\right)_{{|r=0}}.
\end{align*}
By the chain rule we have that
\begin{align}\label{eqdphiv}
d_{\phi}\left(v^-_{\Omega^i[\phi]}[\mu_0^i\circ (\phi^T)^{(-1)}]\right)_{{|\phi=\phi_0}}[h](t,\xi^r_{\phi_0})& = d_{\phi}\left(v^-_{\Omega^i[\phi]}[\mu_0^i\circ (\phi^T)^{(-1)}](t,\xi^r_{\phi}) \right)_{{|\phi=\phi_0}}[h] \\ \nonumber
& \quad- \nabla v^-_{\Omega^i[\phi_0]}[\mu_0^i\circ (\phi_0^T)^{(-1)}](t, \xi^r_{\phi_0}) \cdot d_{\phi} (\xi^r_{\phi})_{{|\phi=\phi_0}}[h].
\end{align}
Note that 
\[
d_{\phi} (\xi^r_{\phi})_{{|\phi=\phi_0}}[h] = h(\xi)+ rd_{\phi}(\nu_\phi)_{{|\phi=\phi_0}}[h](\xi).
\]
Then, the normal derivative of the second term in the right hand side of \eqref{eqdphiv} is:
 \begin{align}\label{drdv}
& -\frac{d}{dr} \left(\nabla v^-_{\Omega^i[\phi_0]}[\mu_0^i\circ (\phi_0^T)^{(-1)}](t, \xi^r_{\phi_0}) \cdot d_{\phi} (\xi^r_{\phi})_{{|\phi=\phi_0}}[h]\right)_{{|r=0}} \\\nonumber
 &=-\frac{d}{dr}\left( \nabla v^-_{\Omega^i[\phi_0]}[\mu_0^i\circ (\phi_0^T)^{(-1)}](t, \xi^r_{\phi_0}) \cdot (h(\xi)+ rd_{\phi}(\nu_\phi)_{{|\phi=\phi_0}}[h](\xi))\right)_{{|r=0}}\\\nonumber
 &=- h(\xi)^\top  \nabla^2 v^-_{\Omega^i[\phi_0]}[\mu_0^i\circ (\phi_0^T)^{(-1)}](t, {\phi_0(\xi)})\nu_{\phi_0}(\xi) -  \nabla v^-_{\Omega^i[\phi_0]}[\mu_0^i\circ (\phi_0^T)^{(-1)}](t, {\phi_0(\xi)}) \cdot d_{\phi}(\nu_\phi)_{{|\phi=\phi_0}}[h](\xi).
  \end{align}
  Next{,} we consider  the second term in the right hand side of equality \eqref{eqneumtot}, i.e. 
 \begin{align}\label{dpn3}
 - \partial_\phi &N_3[\phi_0,\mu_0^o][h](t,\xi) \\ \nonumber
  &= - d_{\phi}\left({\nu_{\Omega^i[\phi]} \cdot \nabla}  v^+_{ \Omega^o}[\mu_0^o]   \circ \phi^T\right)_{{|\phi=\phi_0}}[h](t,\xi)\\ \nonumber
  &=- h(\xi)^\top  \nabla^2 v^+_{\Omega^o}[\mu_0^o ](t, {\phi_0(\xi)})\nu_{\phi_0}(\xi) -  \nabla v^+_{\Omega^o}[\mu_0^o ](t, {\phi_0(\xi)}) \cdot d_{\phi}(\nu_\phi)_{{|\phi=\phi_0}}[h](\xi){\, .}
  \end{align}
  {We have obtained the last equality in \eqref{dpn3} by differentiating under the integral sign in the single layer potential, which can be done because its kernel is non-singular. Then  the sum of equations \eqref{drdv} and \eqref{dpn3} gives}
  \begin{align}\label{eq:bcsh}
 -\frac{d}{dr} &\left(\nabla v^-_{\Omega^i[\phi_0]}[\mu_0^i\circ (\phi_0^T)^{(-1)}](t, \xi^r_{\phi_0}) \cdot d_{\phi} (\xi^r_{\phi})_{{|\phi=\phi_0}}[h]\right)_{{|r=0}}  - \partial_\phi N_3[\phi_0,\mu_0^o][h](t,\xi) \\\nonumber
&=-h(\xi)^\top  \nabla^2 u^N_{\phi_0}(t, {\phi_0(\xi)})\nu_{\phi_0}(\xi) -  \nabla u^N_{\phi_0}(t, {\phi_0(\xi)}) \cdot d_{\phi}(\nu_\phi)_{{|\phi=\phi_0}}[h](\xi).
  \end{align}
  Recalling that
  \[
  d_{\phi}(\nu_\phi)_{{|\phi=\phi_0}}[h](\xi) = -\left((\nabla_{\partial\Omega^i[\phi_0]}(h \circ \phi_0^{-1}))\circ \phi_0\nu_{\phi_0}\right) (\xi),
  \]
  see e.g. Costabel and Le Lou\"er \cite[Lemma 4.3]{CoLe12a},
  we {see} that \eqref{eq:bcsh} becomes
  \begin{equation}\label{eq:bcsh2}
  -h(\xi)^\top  \nabla^2 u^N_{\phi_0}(t, {\phi_0(\xi)})\nu_{\phi_0}(\xi) +  \nabla u^N_{\phi_0}(t, {\phi_0(\xi)}) \cdot \left((\nabla_{\partial\Omega^i[\phi_0]}(h \circ \phi_0^{-1}))\circ \phi_0\nu_{\phi_0}\right)  (\xi).
  \end{equation}
 We  claim that {the expression above  is the    Neumann boundary condition for $\tilde{v}^N$ (pulled back to $[0,T] \times \partial \Omega^i$). In order to verify this claim and conclude the proof,  we have to show that the sum of the remaining terms---the first in the right hand sides of \eqref{eqneumtot} and the normal derivative of the first in the right hand side of \eqref{eqdphiv}---is zero: that is }
 
\[
  \frac{d}{dr}\left(d_{\phi}\left( v^-_{\Omega^i[\phi]}[\mu_0^i\circ (\phi^T)^{(-1)}](t,\xi^r_{\phi}) \right)_{{|\phi=\phi_0}}[h]\right)_{{|r=0}}- \partial_\phi N_1[\phi_0,\mu_0^i][h](t,\xi) =0.
\]
 Let us compute
 \begin{align*}
 \frac{d}{dr} \Big(d_{\phi}&\left(v^-_{\Omega^i[\phi]}[\mu_0^i\circ (\phi^T)^{(-1)}](t,\xi^r_{\phi}) \right)_{{|\phi=\phi_0}}[h]\Big) \\
 &=
\!d_\phi\left(\int_0^t \int_{\partial\Omega^i}\nu_\phi(\xi)\cdot \nabla S_n(t-\tau,\xi^r_\phi -\phi(s))\mu^i_0(\tau,s)\tilde{\sigma}_n[\phi](s)d\sigma_sd\tau\right)_{{|\phi=\phi_0}}[h].
 \end{align*}
 Note that the kernel of the  {integral operator above is smooth and non-singular. For this reason, we can exchange the order of the differentials and the differentiation with respect to $r$, and take the latter under integral sign.} We now want to compute the limit as $r\to 0^+$ and show that it converges to 
 \[
\partial_\phi N_1[\phi_0,\mu_0^i][h](\xi) =d_\phi \left(\int_0^t \int_{\partial\Omega^i}\nu_\phi(\xi)\cdot \nabla S_n(t-\tau,\phi(\xi) -\phi(s))\mu^i_0(\tau,s)\tilde \sigma_n[\phi](s)d\sigma_sd\tau\right)_{{|\phi=\phi_0}}[h].
 \]
 To do so, we first use the  trick of adding and subtracting {{the} density $\mu^i_0(t,\cdot)$ {with  time $t$ fixed}. Indeed,}  heat potentials with  {time-independent} densities coincide---up to the sign---with their harmonic counterpart (see, e.g., \cite[Lemmas A3, A4]{Lu19}). {It should be noted that for the former observation to hold,  the layer heat potentials must be defined by integrating over the time interval  {$]-\infty,t[$ instead of $]0,t[$}.} The advantage of this decomposition is that it yields a less singular integral, plus a term where known results for harmonic potentials can be directly applied. Namely, {if we still denote by $\mu^i_0$ its extension by zero for negative times}, we have:
 \begin{align*}
 \int_0^t& \int_{\partial\Omega^i}\nu_\phi(\xi)\cdot \nabla S_n(t-\tau,\xi^r_\phi -\phi(s))\mu^i_0(\tau,s)\tilde \sigma_n[\phi](s)d\sigma_sd\tau \\
 &= \int_{-\infty}^t \int_{\partial\Omega^i}\nu_\phi(\xi)\cdot \nabla S_n(t-\tau,\xi^r_\phi -\phi(s))(\mu^i_0(\tau,s)-\mu^i_0(t,s))\tilde \sigma_n[\phi](s)d\sigma_sd\tau\\
 &\quad + \int_{-\infty}^t \int_{\partial\Omega^i}\nu_\phi(\xi)\cdot \nabla S_n(t-\tau,\xi^r_\phi -\phi(s)) \mu^i_0(t,s)\tilde \sigma_n[\phi](s)d\sigma_sd\tau\\
  &= \int_{-\infty}^t \int_{\partial\Omega^i}\nu_\phi(\xi)\cdot \nabla S_n(t-\tau,\xi^r_\phi -\phi(s))(\mu^i_0(\tau,s)-\mu^i_0(t,s))\tilde \sigma_n[\phi](s)d\sigma_sd\tau\\
 &\quad -  \int_{\partial\Omega^i}\nu_\phi(\xi)\cdot \nabla \tilde S_n(\xi^r_\phi -\phi(s)) \mu^i_0(t,s)\tilde \sigma_n[\phi](s)d\sigma_s
 \end{align*}
 where 
 \[
 \tilde S_n(x) := \begin{cases}
 \frac{1}{s_2}\log|x| &\quad   \mbox{ if }  n = 2,\\
 \frac{1}{s_n(2-n)}|x|^{2-n} &\quad \mbox{ if } n \in \mathbb{N}\setminus \{1,2\}
 \end{cases} \qquad \forall x \in \mathbb{R}^n\setminus \{0\}
 \]
  is the standard fundamental solution of the Laplace operator and $s_n$ is the $(n-1)$-dimensional measure of the unit sphere in $\mathbb{R}^n$. 
The convergence    as $r \to 0^+$ 
of
\[
 d_{\phi} \left(-  \int_{\partial\Omega^i}\nu_\phi(\xi)\cdot \nabla \tilde S_n(\xi^r_\phi -\phi(s)) \mu^i_0(t,s)\tilde \sigma_n[\phi](s)d\sigma_s\right)_{{|\phi=\phi_0}}[h]
\]
to 
\begin{align*}
& d_{\phi} \left(-  \int_{\partial\Omega^i}\nu_\phi(\xi)\cdot \nabla \tilde S_n(\phi(\xi) -\phi(s)) \mu^i_0(t,s)\tilde \sigma_n[\phi](s)d\sigma_s\right)_{{|\phi=\phi_0}}[h] \\
&\qquad  { = d_{\phi} \left(  \int_{-\infty}^t \int_{\partial\Omega^i}\nu_\phi(\xi)\cdot \nabla S_n(t-\tau,\phi(\xi) -\phi(s)) \mu^i_0(t,s)\tilde \sigma_n[\phi](s)d\sigma_sd\tau\right)_{{|\phi=\phi_0}}[h]}
\end{align*}
follows by Potthast \cite[(5.9)]{Po96a}. Note that Potthast \cite[(5.9)]{Po96a} treats only the case {where} $n=3$, but its argument does not depend on the dimension and can be straightforwardly extended to {any dimension $n \geq 2$}. Then it remains to consider the term 
\begin{equation}\label{eq:diff1}
 d_{\phi} \left(  \int_{-\infty}^t \int_{\partial\Omega^i}\nu_\phi(\xi)\cdot \nabla S_n(t-\tau,\xi^r_\phi -\phi(s))(\mu^i_0(\tau,s)-\mu^i_0(t,s))\tilde \sigma_n[\phi](s)d\sigma_sd\tau\right)_{{|\phi=\phi_0}}[h].
\end{equation}
The convergence of this term to 
\[
 d_{\phi} \left(  \int_{-\infty}^t \int_{\partial\Omega^i}\nu_\phi(\xi)\cdot \nabla S_n(t-\tau,\phi(\xi) -\phi(s))(\mu^i_0(\tau,s)-\mu^i_0(t,s))\tilde \sigma_n[\phi](s)d\sigma_sd\tau\right)_{{|\phi=\phi_0}}[h]
\]
as $r \to 0^+$
is postponed in Lemma \ref{lemA1} for the sake of brevity.

\end{proof}

\subsection{Comparison with Chapko, Kress{,} and Yoon  \cite{ChKrYo99}  }\label{ss:comparisonneum}

{Let the assumptions and the notation of Theorem \ref{cor:derneu} hold.
In order to compare our result with that of  Chapko, Kress and Yoon \cite{ChKrYo99} we consider the specific case where $f^i=0$, since this is the case that they analyze.   We note that,  if we assume homogeneous boundary conditions on the hole and we differentiate the boundary condition 
\[
\frac{\partial}{\partial \nu_{\Omega^i[\phi_0]}}u^N_{\phi_0}=0
\]
 in the direction along the tangential component of $h\circ \phi_0^{(-1)}$, namely we differentiate in the direction of
\[
\left(h\circ \phi_0^{(-1)}\right)_{\partial\Omega^i[\phi_0]} := \left(h\circ \phi_0^{(-1)}\right) - \left(\left(h\circ \phi_0^{(-1)}\right) \cdot \nu_{\Omega^i[\phi_0]}\right) \nu_{\Omega^i[\phi_0]}
\] 
we compute {
\begin{align}\label{tgne}
&\left(h\circ \phi_0^{(-1)}\right)_{\partial\Omega^i[\phi_0]} \cdot \nabla_{\partial\Omega^i[\phi_0]}\left(\nu_{\Omega^i[\phi_0] }\cdot \nabla u^N_{\phi_0}\right) \\ \nonumber 
&\qquad= \left(h\circ \phi_0^{(-1)}\right)^\top_{\partial\Omega^i[\phi_0]} \nabla^2u^N_{\phi_0} \nu_{\Omega^i[\phi_0]}+ \nabla u^N_{\phi_0}\cdot \left(\nabla_{ \partial \Omega^i[\phi_0]} \nu_{\Omega^i[\phi_0]} \right)  \left(h\circ \phi_0^{(-1)}\right)_{\partial\Omega^i[\phi_0]} =0.
\end{align}
Then we note that since 
\[
\nabla_{ \partial \Omega^i[\phi_0]} \left(\nu_{\Omega^i[\phi_0]} \cdot  \left(h\circ \phi_0^{(-1)}\right)_{\partial\Omega^i[\phi_0]}  \right)=0,
\]}
one has
\[
\left(\nabla_{ \partial \Omega^i[\phi_0]} \nu_{\Omega^i[\phi_0]} \right)  \left(h\circ \phi_0^{(-1)}\right)_{\partial\Omega^i[\phi_0]}  = - \nabla_{ \partial \Omega^i[\phi_0]}  \left(h\circ \phi_0^{(-1)}\right)_{\partial\Omega^i[\phi_0]}   \nu_{\Omega^i[\phi_0]}.
\]
Then \eqref{tgne} becomes
\[
\left(h\circ \phi_0^{(-1)}\right)^\top_{\partial\Omega^i[\phi_0]} \nabla^2u^N_{\phi_0} \nu_{\Omega^i[\phi_0]}- \nabla u^N_{\phi_0}\cdot \nabla_{ \partial \Omega^i[\phi_0]}  \left(h\circ \phi_0^{(-1)}\right)_{\partial\Omega^i[\phi_0]}   \nu_{\Omega^i[\phi_0]}=0\, .
\]
Also, since $ \nabla_{\partial\Omega^i[\phi_0]}|  \nu_{\Omega^i[\phi_0]}|^2 =0$, it is readily seen that 
\[
 \nabla_{\partial\Omega^i[\phi_0]}\left(\left((h \circ \phi_0^{(-1)})\cdot  \nu_{\Omega^i[\phi_0]}\right) \nu_{\Omega^i[\phi_0]}\right)  \nu_{\Omega^i[\phi_0]}= \nabla_{\partial\Omega^i[\phi_0]}\left((h \circ \phi_0^{(-1)})\cdot  \nu_{\Omega^i[\phi_0]}\right).
  \]
}
{Then in the boundary conditions on $[0,T] \times  \phi_0(\partial\Omega^i)$ of system  \eqref{shapesysneum} 
 only  the normal direction of {$h\circ \phi_0^{(-1)}$} counts, and we can rewrite it as:}
      \begin{align}\label{eq:bcnew}
   \frac{\partial}{\partial \nu_{\Omega^i[\phi_0]}} u =- \left(\left(h\circ \phi_0^{(-1)}\right) \cdot \nu_{\Omega^i[\phi_0]}\right) \nu_{\Omega^i[\phi_0]}^\top \nabla^2u^N_{\phi_0} \nu_{\Omega^i[\phi_0]}+\nabla u^N_{\phi_0} \cdot  \left(\nabla_{\partial\Omega^i[\phi_0]}\left((h \circ \phi_0^{(-1)})\cdot  \nu_{\Omega^i[\phi_0]}\right)  \right).
  \end{align}
   See also   Henrot and Pierre \cite[p. 227]{HePi18} for more detailed computations.
 Now recall   the following decomposition of the Laplacian on $\partial\Omega^i[\phi_0]$: 
\[
\Delta = \Delta_{\partial\Omega^i[\phi_0]}+ H\frac{\partial}{\partial \nu_{\Omega^i[\phi_0]}}+\frac{\partial^2}{\partial \nu_{\Omega^i[\phi_0]}^2},
\]
where $\Delta_{\partial\Omega^i[\phi_0]}$ is the Laplace-Beltrami operator on $\partial\Omega^i[\phi_0]$, $H$ is the mean curvature {of}  $\partial\Omega^i[\phi_0]$ (i.e. the sum of the principal curvatures), and $\frac{\partial^2}{\partial \nu_{\Omega^i[\phi_0]}^2} u : = \nu_{\Omega^i[\phi_0]}^\top \nabla^2u\,  \nu_{\Omega^i[\phi_0]} $ (see Henrot and Pierre \cite[(5.59)]{HePi18}). 
Then,  recalling  again that we are considering homogeneous Neumann boundary conditions on $\partial\Omega^i[\phi_0]$ and $u^N_{\phi_0}$ is a solution of the heat equation {and is} $C^{\frac{2+\alpha}{2};2+\alpha}$-smooth up to the boundary, we get:
\[
\nu_{\Omega^i[\phi_0]}^\top \nabla^2u^N_{\phi_0} \nu_{\Omega^i[\phi_0]} = \partial_t u^N_{\phi_0}-\Delta_{\partial \Omega^i[\phi_0]}u^N_{\phi_0} \qquad \mbox{ on } [0,T]\times  \partial\Omega^i[\phi_0].
\]
Then we have that the above boundary condition  \eqref{eq:bcnew}  can also be written as 
     \begin{align*}
   \frac{\partial}{\partial \nu_{\Omega^i[\phi_0]}} u = \left(\left(h\circ \phi_0^{(-1)}\right) \cdot \nu_{\Omega^i[\phi_0]}\right)\Delta_{\partial \Omega^i[\phi_0]}u^N_{\phi_0}& -  \left(\left(h\circ \phi_0^{(-1)}\right) \cdot \nu_{\Omega^i[\phi_0]}\right)  \partial_t u^N_{\phi_0} \\
&+\nabla u^N_{\phi_0} \cdot  \left(\nabla_{\partial\Omega^i[\phi_0]}\left((h \circ \phi_0^{(-1)})\cdot  \nu_{\Omega^i[\phi_0]}\right)  \right),
  \end{align*}
  that generalizes the $2$-dimensional result in Chapko, Kress and Yoon \cite[Thm. 2.1]{ChKrYo99}. That is we have the following.

\begin{corollary}
Let $\alpha \in \mathopen]0,1[$ and $T>0$. Let $\Omega^{o}$, $\Omega^{i}$ be as in \eqref{introsetconditions}.  Let $  (f^o,f^i) \in \mathcal{X}_0^{\frac{\alpha}{2}; \alpha}$   and $f^i=0$. Let $\phi_0\in \mathcal{A}^{\Omega^{o}}_{\partial\Omega^{i}}$. Assume that $\phi_0(\partial\Omega^i)$ is of class $C^{2,\alpha}$ and  $h \circ \phi_0^{-1}\in C^{2,\alpha}(\phi_0(\partial\Omega^i),\mathbb{R}^n)$.
Then
\[
d_\phi \Big( {u^N_{\phi}}_{|[0,T] \times \partial \Omega^o}\Big)_{{|\phi=\phi_0}} [h]=  \tilde{v}^N \qquad \text{on}\ [0,T]\times \partial \Omega^{o}\, ,
\]
where $\tilde{v}^N$ is the unique solution to problem 
\[
\begin{cases}
    \partial_t u - \Delta u = 0 & \text{ in } ]0,T] \times (\Omega^{o} \setminus \overline{\Omega^{i}[\phi_0]}), 
    \\
\frac{\partial}{\partial \nu_{\Omega^o}}u = 0&  \mbox{ on }  [0,T] \times \partial \Omega^{o}, 
    \\
   \frac{\partial}{\partial \nu_{\Omega^i[\phi_0]}} u  =  \left(\left(h\circ \phi_0^{(-1)}\right) \cdot \nu_{\Omega^i[\phi_0]}\right)\Delta_{\partial \Omega^i[\phi_0]}u^N_{\phi_0} -  \left(\left(h\circ \phi_0^{(-1)}\right) \cdot \nu_{\Omega^i[\phi_0]}\right)  \partial_t u^N_{\phi_0} \\
\hspace{5cm}+\nabla u^N_{\phi_0} \cdot  \left(\nabla_{\partial\Omega^i[\phi_0]}\left((h \circ \phi_0^{(-1)})\cdot  \nu_{\Omega^i[\phi_0]}\right)  \right)&  \mbox{ on }  [0,T] \times \partial \Omega^{i}[\phi_0],
    \\
    u(0,\cdot)=0 &  \text{ in } \overline{\Omega^{o}} \setminus \Omega^{i}[\phi_0]\, .
    \end{cases}
\]
\end{corollary} 

\appendix
\section{A convergence lemma}

{In this appendix, we prove the following result that we use in the proof of Theorem \ref{cor:derneu}.}
\begin{lemma}\label{lemA1}
Under the assumption and notation of {Theorem} \ref{cor:derneu} and its proof {and for all $(t,\xi) \in [0,T] \times \partial\Omega^i$}, we have that
\begin{equation}\label{eq:intdiffe}
 d_{\phi} \left(   \int_{-\infty}^t \int_{\partial\Omega^i}\nu_\phi(\xi)\cdot \nabla S_n(t-\tau,\xi^r_\phi -\phi(s))(\mu^i_0(\tau,s)-\mu^i_0(t,s))\tilde \sigma_n[\phi](s)d\sigma_sd\tau\right)_{{|\phi=\phi_0}}[h]
\end{equation}
converges   to 
\begin{equation}\label{eq:intdiffe2}
 d_{\phi} \left(  \int_{-\infty}^t \int_{\partial\Omega^i}\nu_\phi(\xi)\cdot \nabla S_n(t-\tau,\phi(\xi) -\phi(s))(\mu^i_0(\tau,s)-\mu^i_0(t,s))\tilde \sigma_n[\phi](s)d\sigma_sd\tau\right)_{{|\phi=\phi_0}}[h]
\end{equation}
as $r \to 0^+$. 
\end{lemma}
\begin{proof}
{Let $(t,\xi) \in [0,T] \times \partial\Omega^i$}. First note that 
\begin{align}\label{eq:vitali1}
 d_{\phi} &\left(   \int_{-\infty}^t \int_{\partial\Omega^i}\nu_\phi(\xi)\cdot \nabla S_n(t-\tau,\xi^r_\phi -\phi(s))(\mu^i_0(\tau,s)-\mu^i_0(t,s))\tilde \sigma_n[\phi](s)d\sigma_sd\tau\right)_{{|\phi=\phi_0}}[h] \\ \nonumber 
 =& d_{\phi} (\nu_\phi)_{{|\phi=\phi_0}}[h](\xi) \cdot\int_{-\infty}^t \int_{\partial\Omega^i}  \nabla S_n(t-\tau,\xi^r_{\phi_0} -\phi_0(s))(\mu^i_0(\tau,s)-\mu^i_0(t,s))\tilde \sigma_n[\phi_0](s)d\sigma_sd\tau\\ \nonumber
 &+\nu_{\phi_0}(\xi)\cdot \int_{-\infty}^t\int_{\partial\Omega^i}  d_{\phi}\Big( \nabla S_n(t-\tau,\xi^r_\phi -\phi(s))\tilde \sigma_n[\phi](s)\Big)_{{|\phi=\phi_0}}[h](\mu^i_0(\tau,s)-\mu^i_0(t,s)) d\sigma_sd\tau.
 \end{align}
 Note that we can differentiate under integral sign since the kernel is smooth and non-singular. We then start to consider the second integral. We recall that
{ \begin{equation}\label{DSn}
 \nabla S_n(\rho,y) = -\frac{y}{2(4\pi)^\frac{n}{2} \rho^{\frac{n}{2}+1}}e^{-\frac{|y|^2}{4\rho}} \qquad \forall (\rho,y) \in \mathopen]0,+\infty\mathclose[\times \mathbb{R}^n  .
 \end{equation}
 Then for all $(\tau,s) \in \mathopen]-\infty,t\mathclose[ \times \partial\Omega^i$:}
 \[
\begin{split}
& d_{\phi}\Big( \nabla S_n(t-\tau,\xi^r_\phi -\phi(s))\tilde \sigma_n[\phi](s)\Big)_{{|\phi=\phi_0}}[h]\\
&\qquad =
-\frac{ h(\xi)+ rd_{\phi}(\nu_\phi)_{{|\phi=\phi_0}}[h](\xi)-h(s)}{2(4\pi)^\frac{n}{2}(t-\tau)^{\frac{n}{2}+1}} e^{-\frac{|\xi^r_{\phi_0} -\phi_0(s)|^2}{4(t-\tau)}}\tilde \sigma_n[\phi_0](s)\\
&\qquad\quad +\frac{\xi^r_{\phi_0} -\phi_0(s)}{4(4\pi)^{\frac{n}{2}}(t-\tau)^{\frac{n}{2}+2}}e^{-\frac{|\xi^r_{\phi_0} -\phi_0(s)|^2}{4(t-\tau)}} (\xi^r_{\phi_0} -\phi_0(s))\cdot(h(\xi)+ rd_{\phi}(\nu_\phi)_{{|\phi=\phi_0}}[h](\xi)-h(s))\,\tilde \sigma_n[\phi_0](s)\\
&\qquad\quad -\frac{\xi^r_{\phi_0} -\phi_0(s)}{2(4\pi)^\frac{n}{2}(t-\tau)^{\frac{n}{2}+1}} e^{-\frac{|\xi^r_{\phi_0} -\phi_0(s)|^2}{4(t-\tau)}}d_{\phi} (\tilde\sigma_n[\phi])_{{|\phi=\phi_0}}[h](s).
\end{split}
\] 
We now want to estimate all the terms in the previous equality proving that they are uniformly integrable in {$\mathopen]-\infty,t\mathclose[ \times \partial\Omega^i$} for $\xi \in \partial\Omega^i$, $r>0$ sufficiently small. We estimate  the second, the others can be estimated exactly in the same way. Note that we can prove that 
\begin{equation}\label{est1}
| h(\xi)+ rd_{\phi}(\nu_\phi)_{{|\phi=\phi_0}}[h](\xi)-h(s)| \leq C |\xi^r_{\phi_0} -\phi_0(s)|
\end{equation}
for some constant $C>0$ independent from $r,s$ and $\xi$. Indeed by Lemma \ref{lemma change of variable} (ii), by the regularity of $h$ and by the bi-Lipschitzianity of $\phi_0$, we have that there exists $C_1,C_2$ independent from $r,s$ and $\xi$ such that
\begin{align*}
| h(\xi)+ rd_{\phi}(\nu_\phi)_{{|\phi=\phi_0}}[h](\xi)-h(s)| &\leq  | h(\xi)-h(s)| + r |d_{\phi}(\nu_\phi)_{{|\phi=\phi_0}}[h](\xi)|\\
&\leq C_1 (|\xi -s|+r)\\
&\leq  C_2 (|\phi_0(\xi) -\phi_0(s)|+r){.}
\end{align*}
{Moreover,} {if  
\[
C_{\phi_0} :=\sup \left\{\frac{|(\phi_0(\xi) -\phi_0(s))\cdot \nu_{\phi_0}(\xi)|}{|\phi_0(\xi) -\phi_0(s)|^2} :  s,\xi \in\partial\Omega^i, s\neq \xi\right\} < +\infty,
\]
see e.g. \cite[Lem. 2.76]{DaLaMu21}, then}
\begin{align*}
 |\xi^r_{\phi_0} -\phi_0(s)|^2&=|\phi_0(\xi)+r\nu_{\phi_0}(\xi)-\phi_0(s)|^2 \\
 &= |\phi_0(\xi) -\phi_0(s)|^2+r^2+2r(\phi_0(\xi) -\phi_0(s))\cdot \nu_{\phi_0}(\xi) \\
&\geq  |\phi_0(\xi) -\phi_0(s)|^2+r^2 - 2r{C_{\phi_0}}|\phi_0(\xi) -\phi_0(s)|^2 \\
& \geq   (1-2r{C_{\phi_0}})|\phi_0(\xi) -\phi_0(s)|^2+r^2\\
&\geq \frac12 \left(\sqrt{1-2r{C_{\phi_0}}}|\phi_0(\xi) -\phi_0(s)|+r \right)^2 
\end{align*}
provided that {$r<\frac{1}{2C_{\phi_0}}$}.
 Accordingly \eqref{est1} holds.
Therefore for some $C_3$ independent  from $r,s$ and $\xi$ we have:
{\begin{align*}
&\Bigg|\frac{\xi^r_{\phi_0} -\phi_0(s)}{4(4\pi)^{\frac{n}{2}}(t-\tau)^{\frac{n}{2}+2}}e^{-\frac{|\xi^r_{\phi_0} -\phi_0(s)|^2}{4(t-\tau)}} (\xi^r_{\phi_0} -\phi_0(s))\cdot(h(\xi)+ rd_{\phi}(\nu_\phi)_{{|\phi=\phi_0}}[h](\xi)-h(s))\,\\
& \qquad \qquad \qquad \qquad \qquad \qquad \qquad \qquad \qquad \qquad \qquad \qquad \times \tilde \sigma_n[\phi_0](s)(\mu^i_0(\tau,s)-\mu^i_0(t,s)) \Bigg| \\
&\qquad\qquad\leq C_3  \|\mu_0^i\|_{C^{\frac{\alpha}{2},\alpha}([0,T]\times\partial\Omega^i)}\frac{ |\xi^r_{\phi_0} -\phi_0(s)|^3}{(t-\tau)^{\frac{n}{2}+2}}e^{-\frac{|\xi^r_{\phi_0} -\phi_0(s)|^2}{4(t-\tau)}}(t-\tau)^{\frac{\alpha}{2}}\tilde \sigma_n[\phi_0](s).
\end{align*}}
Now, first changing the variable in the space domain and then using the change of variable 
\[
u=\frac{|\phi_0(\xi)+r \nu_{\Omega^i[\phi_0]}(\phi_0(\xi))-y|^2}{t-\tau} 
\]
we get: 
\begin{align*}
&\int_{-\infty}^t \int_{\partial\Omega^i}  \frac{ |\xi^r_{\phi_0} -\phi_0(s)|^3}{(t-\tau)^{\frac{n-\alpha}{2}+2}}e^{-\frac{|\xi^r_{\phi_0} -\phi_0(s)|^2}{4(t-\tau)}} \tilde \sigma_n[\phi_0](s)d\sigma_sd\tau\\
&\qquad\qquad= \int_{-\infty}^t \int_{\partial\Omega^i[\phi_0]}  \frac{ |\phi_0(\xi)+r \nu_{\Omega^i[\phi_0]}(\phi_0(\xi)) -y|^3}{(t-\tau)^{\frac{n-\alpha}{2}+2}}e^{-\frac{|\phi_0(\xi)+r \nu_{\Omega^i[\phi_0]}(\phi_0(\xi)) -y|^2}{4(t-\tau)}} d\sigma_yd\tau\\
&\qquad\qquad=  \int_{\partial\Omega^i[\phi_0]}  \frac{1}{ |\phi_0(\xi)+r \nu_{\Omega^i[\phi_0]}(\phi_0(\xi)) -y|^{n-1-\alpha}} \ d\sigma_y\int_0^{+\infty}u^{\frac{n-\alpha}{2}}e^{-\frac u4}  du.
\end{align*}
We can conclude by the Vitali convergence Theorem that the second integral in the right hand side of \eqref{eq:vitali1} is continuous as $r\to 0^+$  Indeed, by the equality above, the uniform integrability of the left hand side in the product space $\mathopen]-\infty,t\mathclose[\times \partial\Omega^i$ follows by the uniform integrability of the first integral in the right hand side that can be found in \cite[Lem. 4.21]{DaLaMu21}. With the same estimates it can also be  proved that the first integral in the right hand side of \eqref{eq:vitali1} is also continuous as $r\to 0^+$.

 {In order to conclude the proof, we need to show that one can differentiate under integral sign in the term
 \[
 d_{\phi} \left(   \int_{-\infty}^t \int_{\partial\Omega^i}\nu_\phi(\xi)\cdot \nabla S_n(t-\tau,\phi(\xi) -\phi(s))(\mu^i_0(\tau,s)-\mu^i_0(t,s))\tilde \sigma_n[\phi](s)d\sigma_sd\tau\right)_{{|\phi=\phi_0}}[h].
\]
so that \eqref{eq:intdiffe} actually converges to  \eqref{eq:intdiffe2} as  $r\to 0^+$. We sketch the main ideas of the proof.  It
clearly suffices to  show that one can differentiate under integral sign in the term
 \begin{align*}
\nu_{\phi_0}(\xi)\cdot & d_{\phi} \left(   \int_{-\infty}^t \int_{\partial\Omega^i} \nabla S_n(t-\tau,\phi(\xi) -\phi(s))(\mu^i_0(\tau,s)-\mu^i_0(t,s))\tilde \sigma_n[\phi](s)d\sigma_sd\tau\right)_{{|\phi=\phi_0}}[h]\\
&=\nu_{\phi_0}(\xi)\cdot  \frac{d}{d\varepsilon} \Bigg(   \int_{-\infty}^t \int_{\partial\Omega^i} \nabla S_n(t-\tau,\phi_0(\xi) -\phi_0(s)+\varepsilon(h(\xi)-h(s)))\\
&\qquad \qquad \qquad \qquad\qquad\qquad \qquad\times (\mu^i_0(\tau,s)-\mu^i_0(t,s))\tilde \sigma_n[\phi_0+\varepsilon h](s)d\sigma_sd\tau\Bigg)_{|\varepsilon=0}.
\end{align*}
To do that, we need to show that the maps
\begin{align}\label{rdrS}
&\varepsilon \mapsto \nabla S_n(t-\tau,\phi_0(\xi) -\phi_0(s)+\varepsilon(h(\xi)-h(s)))(\mu^i_0(\tau,s)-\mu^i_0(t,s))\tilde \sigma_n[\phi_0+ \varepsilon h](s), \\ \nonumber
&\varepsilon\mapsto  \frac{d}{d\varepsilon}\Big(\nabla S_n(t-\tau,\phi_0(\xi) -\phi_0(s)+\varepsilon(h(\xi)-h(s)))\tilde \sigma_n[\phi_0+\varepsilon h](s)\Big)(\mu^i_0(\tau,s)-\mu^i_0(t,s))
\end{align}
are uniformly dominated by functions in $L^1({]-}\infty,t]\times \partial\Omega^i)$ for $\varepsilon$ in a neighborhood of $0$. We analyze only the second map in \eqref{rdrS} since the first one can be dominated in the same way. Recalling \eqref{DSn} we have
\begin{align}\label{drS2}
&\frac{d}{d\varepsilon}\Big(\nabla S_n(t-\tau,\phi_0(\xi) -\phi_0(s)+\varepsilon(h(\xi)-h(s))) \tilde \sigma_n[\phi_0+\varepsilon h](s)\Big)\\ \nonumber
&\qquad =
-\frac{h(\xi)-h(s)}{2(4\pi)^\frac{n}{2}(t-\tau)^{\frac{n}{2}+1}} e^{-\frac{|\phi_0(\xi) -\phi_0(s)+\varepsilon(h(\xi)-h(s))|^2}{4(t-\tau)}} \tilde\sigma_n[\phi_0+\varepsilon h](s)\\ \nonumber
&\qquad\quad +\frac{\phi_0(\xi) -\phi_0(s)+\varepsilon(h(\xi)-h(s))}{4(4\pi)^\frac{n}{2}(t-\tau)^{\frac{n}{2}+2}}e^{-\frac{|\phi_0(\xi) -\phi_0(s)+\varepsilon(h(\xi)-h(s))|^2}{4(t-\tau)}} \\ \nonumber
&\qquad \qquad\qquad\times (\phi_0(\xi) -\phi_0(s)+\varepsilon(h(\xi)-h(s)))\cdot(h(\xi)-h(s))\,\tilde \sigma_n[\phi_0+\varepsilon h](s)\\ \nonumber
&\qquad\quad -\frac{\phi_0(\xi) -\phi_0(s)+\varepsilon(h(\xi)-h(s))}{2(4\pi)^\frac{n}{2}(t-\tau)^{\frac{n}{2}+1}} e^{-\frac{|\phi_0(\xi) -\phi_0(s)+\varepsilon(h(\xi)-h(s))|^2}{4(t-\tau)}} \frac{d}{d\varepsilon}\tilde\sigma_n[\phi_0+\varepsilon h] (s).
\end{align}
Note that for a sufficiently small $|\varepsilon|$ we have
\[
|\phi_0(\xi) -\phi_0(s)+\varepsilon (h(\xi)-h(s))| \leq |\phi_0(\xi) -\phi_0(s)| + |\varepsilon| |h(\xi)-h(s)| \leq (\|\nabla \phi_0\|_\infty+\|\nabla h\|_\infty)|\xi-s|
\]
and for some $C_4$ independent of $\xi,s$
\begin{align*}
|\phi_0(\xi) -\phi_0(s)+\varepsilon(h(\xi)-h(s))| &\geq  |\phi_0(\xi) -\phi_0(s)| - |\varepsilon| |h(\xi)-h(s)|  \\
&\geq C_4 |\xi-s|-|\varepsilon|\|\nabla h\|_{\infty}|\xi-s|\\
&\geq \frac{C_4}{2} |\xi-s|.
\end{align*}
We estimate only the second term in the right hand side of \eqref{drS2}, and the others can be estimated in the same way. By the previous computations there exists a constant $C_5>0$ such that 
{\begin{align}\label{rdrS3}
&\Bigg|\frac{\phi_0(\xi) -\phi_0(s)+\varepsilon (h(\xi)-h(s))}{4(4\pi)^{\frac{n}{2}}(t-\tau)^{\frac{n}{2}+2}}e^{-\frac{|\phi_0(\xi) -\phi_0(s)+\varepsilon (h(\xi)-h(s))|^2}{4(t-\tau)}}(\mu^i_0(\tau,s)-\mu^i_0(t,s))\\ \nonumber
&\qquad\qquad\qquad \times (\phi_0(\xi) -\phi_0(s)+\varepsilon(h(\xi)-h(s)))\cdot(h(\xi)-h(s))\,\tilde\sigma_n[\phi_0+\varepsilon h](s)\Bigg| \\ \nonumber
&\qquad\leq C_5\|\mu_0^i\|_{C^{\frac{\alpha}{2};\alpha}([0,T]\times \partial\Omega^i)}\frac{|\xi-s|^3}{4(4\pi)^{\frac{n}{2}}(t-\tau)^{\frac{n}{2}+2}}e^{-\frac{C_4|\xi-s|^2}{16(t-\tau)}}|t-\tau|^\frac{\alpha}{2}.
\end{align}}
With the usual change of variable
\[
u = \frac{|\xi-s|^2}{(t-\tau)}
\]
it is easily seen that the right hand side  of \eqref{rdrS3} is integrable in ${]}-\infty,t]\times \partial\Omega^i$ and accordingly we are done.
}

\end{proof}

\section*{Acknowledgment}

The authors are members of the ``Gruppo Nazionale per l'Analisi Matematica, la Probabilit\`a e le loro Applicazioni'' (GNAMPA) of the ``Istituto Nazionale di Alta Matematica'' (INdAM).
The authors acknowledge the support  of the
project funded by the European Union - NextGenerationEU under the National Recovery and
Resilience Plan (NRRP), Mission 4 Component 2 Investment 1.1 - Call PRIN 2022 No. 104 of
February 2, 2022 of Italian Ministry of University and Research; Project 2022SENJZ3 (subject area: PE - Physical Sciences and Engineering) ``Perturbation problems and asymptotics for elliptic differential equations: variational and potential theoretic methods'' { and the support of the  INdAM - GNAMPA Project ``Functional and spectral analysis for differential operators'',  CUP E53C25002010001}. {M.~Dalla Riva also acknowledges the support by MUR (Ministero dell'Universit\`a e della Ricerca) through the PNRR Project QUANTIP - Partenariato Esteso NQSTI - PE00000023 - Spoke 9 - CUP: E63C22002180006. P.~Musolino also
acknowledges the support from EU through the H2020-MSCA-RISE-2020 project EffectFact, Grant agreement ID: 101008140. } 

\end{document}